\documentclass[10pt]{article}
\usepackage[T1]{fontenc}
\usepackage[utf8]{inputenc}
\usepackage{color}
\usepackage{verbatim}
\usepackage{mathtools}
\usepackage{enumitem}
\usepackage{amsmath}
\usepackage{amsthm}
\usepackage{amssymb}
\usepackage{stmaryrd}
\usepackage{authblk}
\usepackage[all]{xy}
\usepackage[bookmarks=false,
 breaklinks=false,pdfborder={0 0 1},backref=false,colorlinks=true]
 {hyperref}
\hypersetup{
 allcolors=blue}

\makeatletter

\usepackage{amsfonts}

\usepackage{orcidlink}

\DeclareMathOperator{\Ul}{Ul}

\DeclareMathOperator{\cl}{cl}
\DeclareMathOperator{\In}{int}

\DeclareMathOperator{\Fi}{Fi}

\newtheorem{theorem}{Theorem}\newtheorem{lemma}[theorem]{Lemma}\newtheorem{corollary}[theorem]{Corollary}\newtheorem{proposition}[theorem]{Proposition}\theoremstyle{definition}
\newtheorem{definition}[theorem]{Definition}\newtheorem{remark}[theorem]{Remark}

\newcommand{\email}[1]{\thanks{#1}}
\setlist[enumerate,1]{label={\(\mathrm{(\arabic*)}\)}, align=left, left=0pt}

\providecommand{\theoremname}{Theorem}
\providecommand{\examplename}{Example}
\@ifundefined{thm}{\theoremstyle{plain}\newtheorem{thm}{\protect\theoremname}}{}
\@ifundefined{example}{\theoremstyle{definition}\newtheorem{example}[thm]{\protect\examplename}}{}

\makeatother

\providecommand\theoremname{Theorem}
\providecommand\examplename{Example}

\title{Duality Theory for Sequent Algebras}
\author[1,2]{Sergio A. Celani\orcidlink{0000-0003-2542-4128}\medskip\email{scelani@exa.unicen.edu.ar}}
\author[1,2]{William Zuluaga\orcidlink{0000-0002-8798-9493}\medskip \email{wizubo@gmail.com}}
\affil[1]{Departamento de Matem\'atica, Facultad de Ciencias Exactas (UNCPBA), Pinto 399, Tandil, Argentina}
\affil[2]{CONICET, Argentina}
\date{}
\begin{document}
\maketitle
\begin{abstract}
Sequent algebras were introduced by D. Vakarelov \cite{VakarelovSequentALg} as a generalization
of contact algebras, replacing the binary contact relation with a
relation between finite sets of regions and regions. In this paper,
we present sequent algebras as Boolean algebras endowed with a finitary
closure operator satisfying four conditions. We study the irreducible
closed sets of these operators, which coincide with Vakarelov's clans,
and use them to develop a topological representation theory for sequent
algebras. We introduce the categories of sequent algebras and sequent
spaces and establish a dual equivalence between them, as well as a
restricted dual equivalence for morphisms preserving the full Boolean
structure. We also characterize sequent congruences in terms of positive
filters and minimally decreasing closed subsets of the associated
sequent spaces. Finally, we show that the category of sequent algebras
with sequent homomorphisms is topological over the category of Boolean
algebras, obtaining completeness and cocompleteness and a complete
lattice structure on the sequent structures over each Boolean algebra.
\end{abstract}


\section{Introduction}

Region-based theories of space, also known as mereotopology, take regions rather than points as the primitive objects of space. A standard algebraic setting for these theories is given by contact algebras, that is, Boolean algebras endowed with a binary contact relation, whose prototypical models are the Boolean algebras $\mathrm{RC}(X)$ of regular closed subsets of a topological space $X$, where two regions are in contact when they have a common point.

In \cite{VakarelovSequentALg}, D. Vakarelov introduced sequent algebras as a generalization of contact algebras. In a sequent algebra, the contact relation is replaced by a relation $a_1,\ldots,a_n\vdash b$ between finite sets of regions and regions, which, in the standard model $\mathrm{RC}(X)$, means that $a_1\cap\cdots\cap a_n\subseteq b$. The relation $\vdash$ satisfies some formal properties of a Tarski consequence relation \cite{T1930}, and the contact relation is recovered by declaring that $a$ and $b$ are in contact when $a,b\nvdash 0$. Vakarelov developed a topological representation theory for sequent algebras, establishing a correspondence between point-free and point-based models of space. Related $n$-ary generalizations of contact, the hypercontact semilattices, were recently studied by P. Lipparini \cite{Lipparini2025}, and a duality for mereotopological spaces based on a notion of mereocompactness was developed by R. Goldblatt and M. Grice \cite{Goldblatt Mereocompactness}.

The aim of this paper is to study sequent algebras from the perspective of closure operators, in the spirit of abstract algebraic logic \cite{F2016}. We show that a sequent algebra can be equivalently described as a Boolean algebra $A$ endowed with a finitary closure operator $C$ satisfying four conditions: $C(X)$ is contained in the filter generated by $X$, $C(\varnothing)=\{1\}$, every $C$-closed set is an upset of $A$, and a join condition holds. This description allows us to apply the general theory of closure systems. In particular, we show that the irreducible closed sets, which coincide with Vakarelov's clans, form a base of the closure system and play the role of points in the associated dual space.

Our main results concern the categorical duality for sequent algebras, the characterization of their congruences, and the categorical structure of the resulting categories. First, we introduce sequent spaces and establish a dual equivalence between the category of sequent algebras with sequent semi-homomorphisms and the category of sequent spaces with sequent semi-morphisms. We also show that this duality restricts to a dual equivalence between the categories whose morphisms preserve the full Boolean structure. Second, we introduce sequent congruences and characterize them in terms of positive filters. We prove that positive filters coincide with the regular filters determined by a natural subordination associated with the sequent relation. Moreover, the lattice of sequent congruences is isomorphic to the lattice of regular filters and dually isomorphic to the lattice of minimally decreasing closed subsets of the associated sequent space. Finally, we show that the category of sequent algebras with sequent homomorphisms is topological over the category of Boolean algebras. This result yields, among other consequences, the existence of discrete and indiscrete functors, completeness and cocompleteness, and a complete lattice structure on the set of sequent structures that can be defined on a fixed Boolean algebra. In contrast, we show that the category of sequent algebras with sequent semi-homomorphisms is not topological over the category of Boolean algebras and join-semi-homomorphisms preserving $1$.

The paper is organized as follows. In Section 2 we fix notation and recall basic facts about Boolean algebras, closure operators, closure systems, and regular closed sets. In Section 3 we recall the definition of sequent algebra given in \cite{VakarelovSequentALg} and prove that sequent algebras are equivalent to Boolean algebras endowed with a finitary closure operator satisfying the conditions $\mathrm{Fi}$, $\mathrm{One}$, $\mathrm{Up}$, and $\mathrm{Jo}$. We also relate sequent algebras to the hypercontact semilattices studied in \cite{Lipparini2025}. In Section 4 we study irreducible closed subsets of Boolean closure spaces. We prove a separation lemma, show that irreducible closed sets form a base of the closure system, characterize them as join-prime closed sets, and identify them with the clans of \cite{VakarelovSequentALg}.

In Section 5 we introduce sequent spaces and study their basic properties. In Section 6 we define the relevant categories of sequent algebras and sequent spaces, construct canonical representations for both types of structures, and establish the corresponding dual equivalences. In Section 7 we study positive filters, regular filters, and sequent congruences. We prove that positive and regular filters coincide and establish the correspondence between sequent congruences, regular filters, and minimally decreasing closed subsets of the dual space. Finally, in Section 8 we prove that the category of sequent algebras with sequent homomorphisms is topological over the category of Boolean algebras. We derive several consequences of this result, including the existence of discrete and indiscrete functors, completeness and cocompleteness, and the fact that the sequent structures on a fixed Boolean algebra form a complete lattice. We conclude by showing that the analogous topologicality result fails for the category of sequent algebras with sequent semi-homomorphisms over the category of Boolean algebras and join-semi-homomorphisms preserving $1$.

\section{Preliminaries}

The set-theoretical complement of a subset $Y\subseteq X$ will be
denoted by $Y^{c}$. Let $Z\subseteq X$ and $W\subseteq Y$. If $f\colon X\to Y$ is a map, we write $f[Z]$ for the direct image of $Z$ under $f$ and $f^{-1}[W]$ for the inverse image of $W$ under $f$.

For a Boolean algebra $\mathbf{A}$, $\Fi(\mathbf{A})$, $\mathrm{Id}(\mathbf{A})$ and $\Ul(\mathbf{A})$ denote the sets
of all filters, ideals and ultrafilters of $\mathbf{A}$, respectively. The universe of $\mathbf{A}$ is denoted by $A$. The filter generated by a
subset $X\subseteq A$ is defined as: 
\[
\mathrm{Fg}(X)\coloneqq\left\{ a\in A:\exists\{b_{1},\ldots,b_{n}\}\subseteq X\left(b_{1}\wedge\ldots\wedge b_{n}\leq a\right)\right\} .
\]
If $X=\left\{ a\right\}$, we write $\mathrm{Fg}(a)$. Note that
$\mathrm{Fg}(a)={\uparrow}a=\{b\in A\colon a\leq b\}$.

The family of open sets of the Stone space $\Ul(\mathbf{A})$ associated to
a Boolean algebra $\mathbf{A}$ is denoted by $\tau$. The Stone mapping $\varphi_{\mathbf{A}}:A\rightarrow\mathcal{P}(\Ul(\mathbf{A}))$
is given by 
\[
\varphi_{\mathbf{A}}(a)\coloneqq\left\{ U\in\Ul(\mathbf{A}):a\in U\right\} .
\]
As is well known, the family $\{\varphi_{\mathbf{A}}(a):a\in A\}$ forms a
field of subsets of $\Ul(\mathbf{A})$, and therefore is a Boolean subalgebra
of the algebra $\mathcal{P}(\Ul(\mathbf{A}))$.

\medskip{}

We recall that a \emph{closure} \emph{operator} on a set $A$ is a
mapping 
\[
\mathsf{\mathrm{C}}:\mathcal{P}(A)\rightarrow\mathcal{P}(A)
\]
satisfying the following conditions: 
\begin{enumerate}
\item[$(\mathrm{C1})$] $X\subseteq\mathrm{C}(X)$, 
\item[$(\mathrm{C2})$] If $X\subseteq Y$, then $\mathrm{C}(X)\subseteq\mathrm{C}(Y)$, 
\item[$(\mathrm{C3})$] $\mathrm{C}(\mathrm{C}(X))\subseteq\mathrm{C}(X)$. 
\end{enumerate}
A closure operator $\mathrm{C}$ is \emph{finitary} if 
\begin{enumerate}
\item[$(\mathrm{C}4)$] $\mathrm{C}(X)=\bigcup\left\{ \mathrm{C}(Y):Y\in\mathcal{P}_{f}(X)\right\}, $ 
\end{enumerate}
where $\mathcal{P}_{f}(X)=\{Y\subseteq X \colon Y\; \text{is finite}\}$.

A \emph{closure system} on a set
$A$ is a family $\mathcal{C}$ of subsets of $A$ such that $A\in\mathcal{C}$
and $\mathcal{C}$ is closed under arbitrary intersections. A subfamily $\mathcal{D}\subseteq \mathcal{C}$ is \emph{upward-directed} if for every $X,Y\in \mathcal{D}$ there exists $Z\in \mathcal{D}$ such that $X\subseteq Z$ and $Y\subseteq Z$. A closure
system $\mathcal{C}$ is \emph{inductive} if it is closed
under unions of nonempty upward-directed subfamilies.

There exists a bijective correspondence between closure operators
and closure systems on $A$. If $\mathrm{C}$ is a closure operator
on $A$, then the family $\mathcal{C}_{\mathrm{C}}=\left\{ X\subseteq A:\mathrm{C}(X)=X\right\} $
is a closure system. The elements of $\mathcal{C}_{\mathrm{C}}$ are
called $\mathrm{C}$-closed subsets of $A$, or simply $\mathrm{C}$-closed subsets, when no confusion arises. If $\mathcal{C}$ is a closure
system on $A$, then the mapping $\mathrm{C}_{\mathcal{C}}:\mathcal{P}(A)\rightarrow\mathcal{P}(A)$
defined by $\mathrm{C}_{\mathcal{C}}(X)=\bigcap\left\{ Y\in\mathcal{C}:X\subseteq Y\right\} $,
for each $X\subseteq A$, is a closure operator. Theorem 1.27 of \cite{F2016} shows that these two correspondences are inverse of each other. Under these
bijection, it is easy to see that a closure operator $\mathrm{C}$
is finitary if and only if $\mathcal{C}_{\mathrm{C}}$ is inductive.

A \emph{closure} \emph{space} is a pair $\left<A,\mathrm{C}\right>$
where $A$ is a non-empty set and $\mathrm{C}$ is a closure operator
on $A$. Given the bijective correspondence between closure operators
and closure systems, we can equivalently represent $\left<A,\mathrm{C}\right>$
as $\left<A,\mathcal{C}\right>,$ where $\mathcal{C}=\mathcal{C}_{\mathrm{C}}.$
When $\mathbf{A}$ is a Boolean algebra, we denote the closure space $\left<A,\mathrm{C}\right>$ by $\left<\mathbf{A},\mathrm{C}\right>$. In such a case, the closure space 
is referred to as a \emph{Boolean closure space}. If $\mathrm{C}$
is finitary, then $\left<\mathbf{A},\mathrm{C}\right>$ is called a \emph{finitary
Boolean closure space}.

Observe that $\mathrm{Fi}(\mathbf{A})$
is a closure system on $A$ and the closure operator associated to $\mathrm{Fi}(\mathbf{A})$
is the filter generated by a set, i.e., $\mathrm{C}_{\mathrm{Fi}(\mathbf{A})}(X)=\mathrm{Fg}(X)$,
with $X\subseteq A$. Therefore, $\left<\mathbf{A},\mathrm{Fg}\right>$ is a finitary
Boolean closure space.

\medskip{}

Let $X$ be a topological space. For each $U\subseteq X$, let $\In U$
and $\cl U$ denote the topological interior and closure of $U$,
respectively. A set $U\subseteq X$ is called \emph{regular closed}
if $U=\cl(\In(U)).$ It is well known that the collection $\mathrm{RC}(X)$
of regular closed subsets of $X$ forms a complete Boolean algebra
where $\emptyset$ is the bottom element, $X$ is the top element,
and where the operations are defined as $-U=\cl(X\setminus U)$, $U\vee V=U\cup V$,
and $U\wedge V=\cl(\In(U\cap V))$. The infinite meets and joins in
$\mathrm{RC}(X)$ are given by $\bigvee U_{i}=\cl(\bigcup U_{i})$
and $\bigwedge U_{i}=\cl(\In(\bigcap U_{i}))$, respectively.

\section{Sequent algebras}

In this section, we recall the definition of sequent algebras as given
by D. Vakarelov \cite{VakarelovSequentALg}, and we propose an alternative
formulation by means of closure operators subject to some additional
conditions.

\begin{definition} \cite{VakarelovSequentALg}\label{def: S-algebras} A sequent Boolean
algebra, or \textbf{$S$}-algebra, is a pair $\left<\mathbf{A},\vdash\right>$,
where $\mathbf{A}$ is a Boolean algebra and $\vdash$ is a subset of $\mathcal{P}_{f}(A)\times A$
satisfying the following conditions: 
\begin{enumerate}
\item $a\vdash a$, 
\item If $X\vdash a$, then $X,b\vdash a$, 
\item If $X\vdash a$ and $X,a\vdash b$, then $X\vdash b$, 
\item $a\vdash b$ if and only if $a\leq b$. 
\item $\emptyset\vdash a$ if and only if $a=1$, 
\item If $X,a\vdash c$ and $X,b\vdash c$, then $X,a\vee b\vdash c$. 
\end{enumerate}
\end{definition}

The relation $\vdash$ can be extended to a relation $\vdash_{f}\subseteq\mathcal{P}(A)\times A$
as follows: 
\[
X\vdash_{f}a\;\text{ iff }\;\exists X_{0}\in \mathcal{P}_{f}(X)\left(X_{0}\vdash a\right).
\]
This allows defining a closure operator $\mathrm{C}_{\vdash}:\mathcal{P}(A)\rightarrow\mathcal{P}(A)$
by 
\[
\mathrm{C}_{\vdash}(X)\coloneqq\left\{ a\in A:X\vdash a\right\} ,
\]
for each $X\subseteq A$. In this work, we will only consider finitary
closure operators. Therefore, it makes no difference whether we write
$\vdash$ or $\vdash_{f}$.

\begin{lemma}\label{props SequentALg} \cite{VakarelovSequentALg}
Let $\left<\mathbf{A},\vdash\right>$ be an $S$-algebra. Then the following
conditions hold: 
\begin{enumerate}
\item If $X,a,b\vdash c$, then $X,a\wedge b\vdash c$. 
\item If $a_{1},\ldots,a_{n}\vdash b$, then $a_{1}\wedge\ldots\wedge a_{n}\leq b$. 
\item If $X,a\vdash b$ and $c\leq a$, then $X,c\vdash b$. 
\item If $X\vdash a$ and $a\leq b$, then $X\vdash b$. 
\item $X\vdash1$. 
\item If $X,a\vdash b$ and $X\vdash a\vee b$, then $X\vdash b$. 
\item If $X_{1},a_{1}\vdash b_{1}$ and $X_{2},a_{2}\vdash b_{2}$, then
$X_{1},X_{2},a_{1}\vee a_{2}\vdash b_{1}\vee b_{2}$. 
\end{enumerate}
\end{lemma} 
\begin{example}
\cite{VakarelovSequentALg} Let $X$ be a topological space. Let $\mathbf{RC}(X)$
be the Boolean algebra of regular closed subsets of $X$. For a finite
subset $\left\{ a_{1},\ldots,a_{n}\right\} \subseteq\mathrm{RC}(X)$
and $b\in\mathrm{RC}(X)$ define a consequence relation in $\mathrm{RC}(X)$
as follows: 
\[
a_{1},\ldots,a_{n}\vdash b\text{ iff }a_{1}\cap\ldots\cap a_{n}\subseteq b.
\]
Then $\langle\mathbf{RC}(X),\vdash\rangle$ is a sequent algebra. We note that
the relation $\vdash$ can be defined as follows: 
\[
a_{1},\ldots,a_{n}\vdash b\text{ iff }b\in\mathrm{Fg}(\{a_{1},\ldots,a_{n}\}),
\]
where $\mathrm{Fg}(\{a_{1},\ldots,a_{n}\})$ is the filter generated
by the finite set $\{a_{1},\ldots,a_{n}\}$ in the Boolean algebra
$\langle \mathcal{P}(X), \cup, \cap, ^{c},\emptyset, X\rangle$. 
\end{example}

We now provide another characterization of sequent algebras using finitary
closure operators.

\begin{proposition} \label{prop: charact closure}Let $\mathbf{A}$ be a Boolean
algebra. Let $\mathsf{C}:\mathcal{P}(A)\rightarrow\mathcal{P}(A)$
be a finitary closure operator on $A$ satisfying the following conditions
for each $X\cup\left\{ a,b\right\} \subseteq A$: 
\begin{enumerate}
\item[$\mathbf{Fi}$] $\mathrm{C}(X)\subseteq\mathrm{Fg}(X)$. 
\item[$\mathbf{One}$] $\mathrm{C}(\emptyset)=\{1\}$.  
\item[$\mathbf{Up}$] $\mathrm{C}(X)\in\mathrm{Up}(A)$. 
\item[$\mathbf{Jo}$] $\mathrm{C}(X,a)\cap\mathrm{C}(X,b)\subseteq\mathrm{C}(X,a\vee b)$. 
\end{enumerate}
Define the relation ${\vdash_{\mathrm{C}}}\subseteq\mathcal{P}_{f}(A)\times A$
given by 
\[
X\vdash_{\mathrm{C}}b\;\text{ iff }\;\exists Y\in\mathcal{P}_{f}(X) (b\in\mathrm{C}(Y)).
\]
Then $\left<A,\vdash_{\mathrm{C}}\right>$ is an $S$-algebra.

Conversely, if $\left<\mathbf{A},\vdash\right>$ is an $S$-algebra, then the
mapping $\mathrm{C}_{\vdash}:\mathcal{P}(A)\rightarrow\mathcal{P}(A)$
given by 
\[
\mathrm{C}_{\vdash}(X)\coloneqq\left\{ a\in A:\exists Y\in\mathcal{P}_{f}(X)(Y\vdash a)\right\} ,
\]
is a finitary closure operator satisfying the conditions $\mathbf{Fi}$,
$\mathbf{One}$, $\mathbf{Up}$, and $\mathbf{Jo}$. Moreover, $\mathrm{C}=\mathrm{C}_{\vdash_{\mathrm{C}}}$
and ${\vdash}={\vdash_{\mathrm{C}_{\vdash}}}$.

\end{proposition} 
\begin{proof}
We prove that $\left<\mathbf{A},\vdash_{\mathrm{C}}\right>$ is an $S$-algebra. To do so, let $a,b\in A$ and let $X\in\mathcal{P}_{f}(A)$. 
\begin{itemize}
\item[(1)] Since $a\in \mathrm{C}(a)$, we have $a\vdash_{\mathrm{C}}a$.
\item[(2)] If $X\vdash_{\mathrm{C}}a$, then $a\in\mathrm{C}(X)\subseteq\mathrm{C}\left(X\cup\left\{ b\right\} \right)$,
i.e., $X,b\vdash_{\mathrm{C}}a$.
\item[(3)] If $X\vdash_{\mathrm{C}}a$ and $X,a\vdash_{\mathrm{C}}b$, then $a\in\mathrm{C}(X)$
and $b\in\mathrm{C}\left(X\cup\left\{ a\right\} \right)$. So \[b\in\mathrm{C}\left(X\cup\left\{ a\right\} \right)\subseteq\mathrm{C}\left(X\cup\mathrm{C}(X)\right)=\mathrm{C}\left(\mathrm{C}(X)\right)=\mathrm{C}(X).\] Therefore, $X\vdash_{\mathrm{C}}b$.
\item[(4)] Now we show 
\[
a\leq b\text{ iff }a\vdash_{C}b.
\]
Let $a\leq b$. As $a\in\mathrm{C(a)},$ by $\mathbf{\mathbf{Up}},$
$b\in\mathrm{C}(a)$, so $a\vdash_{C}b$. Conversely, if $a\vdash_{C}b$,
then $b\in\mathrm{C}(a)\subseteq\mathrm{Fg}(a)$ (by $\mathbf{Fi}$),
thus, $a\leq b$.
\item[(5)] Suppose that $\emptyset\vdash_{\mathrm{C}}a$. Then, $a\in\mathrm{C}(\emptyset)\subseteq\mathrm{Fg}(\emptyset)=\left\{ 1\right\} $,
hence $a=1$. If $a=1$, then $a\in\mathrm{C}(\emptyset)$ by $\mathbf{One}$,
therefore $\emptyset\vdash_{\mathrm{C}}a$. 
\item[(6)] By condition $\mathbf{Jo}$, we have that if $X,a\vdash_{\mathrm{C}}c$
and $X,b\vdash_{\mathrm{C}}c$, then $X,a\vee b\vdash_{\mathrm{C}}c$. 
\end{itemize}
Then $\left<\mathbf{A},\vdash_{\mathrm{C}}\right>$ is an $S$-algebra. Conversely, if $\langle\mathbf{A},\vdash\rangle$ is an $S$-algebra, then it is easy to see that $\mathrm{C}_{\vdash}$ is a finitary closure operator satisfying conditions $\mathbf{Fi}$, $\mathbf{One}$, $\mathbf{Up}$, and $\mathbf{Jo}$. Moreover, it readily follows that $\mathrm{C}=\mathrm{C}_{\vdash_{\mathrm{C}}}$ and ${\vdash} = {\vdash_{\mathrm{C}_{\vdash}}}$.

\end{proof}

By this result, the notion of a sequent algebra is equivalent to that of a Boolean algebra endowed with a closure operator satisfying
the properties $\mathbf{Fi}$, $\mathbf{One}$, $\mathbf{Up}$, and
$\mathbf{Jo}$. Thus, in all that follows, we will consider
a sequent algebra as a pair $\left<\mathbf{A},\vdash\right>$ or a pair $\left<\mathbf{A},\mathrm{C}\right>$. 

\begin{remark}\label{rem: order C} Let $\left<\mathbf{A},\mathrm{C}\right>$
be an $S$-algebra. By the proof of Proposition \ref{prop: charact closure}
we have that $\mathrm{C}(a)=\mathrm{Fg}(a)$, for each $a\in A$.
Thus, the lattice order $\leq$ is determined by $\mathrm{C}$, because
\[
a\leq b\text{ iff }\mathrm{C}(b)\subseteq\mathrm{C}(a).
\]
\end{remark} 
\begin{remark}\label{additivity}
Let $\mathbf{A}$ be a Boolean algebra and let
$\mathrm{C}$ be a closure operator on $A$ satisfying the property
$\mathbf{Jo}$ of Proposition \ref{prop: charact closure}. Then 
\[
\mathrm{C}(X,a)\cap\mathrm{C}(X,b)=\mathrm{C}(X,a\vee b),
\]
for any $X\cup\left\{ a,b\right\} \subseteq A$. Indeed, since $a,b\leq a\vee b$,
we have 
\[
\mathrm{C}(a\vee b)\subseteq\mathrm{C}(a)\cap\mathrm{C}(b)\subseteq\mathrm{C}\left(X,a\right)\cap\mathrm{C}\left(X,b\right).
\]
Hence $a\vee b\in\mathrm{C}\left(X,a\right)\cap\mathrm{C}\left(X,b\right)$.
Since $X\subseteq\mathrm{C}(X)\subseteq\mathrm{C}\left(X,a\right)\cap\mathrm{C}\left(X,b\right)$,
it follows that $X\cup\left\{ a\vee b\right\} \subseteq\mathrm{C}\left(X,a\right)\cap\mathrm{C}\left(X,b\right)$. Therefore,  
\[\mathrm{C}(X\cup\left\{ a\vee b\right\})\subseteq \mathrm{C}(\mathrm{C}\left(X,a\right)\cap\mathrm{C}\left(X,b\right))\subseteq \mathrm{C}\left(X,a\right)\cap\mathrm{C}\left(X,b\right),\] 
which establishes the equality.
\end{remark}

Next, we recall the definition of a hypercontact semilattice (lattice)
introduced in \cite{Lipparini2025}.

\begin{definition}\label{def:hypercontact semilattice}Let $\mathbf{L}=\left<L,\vee,0\right>$
be a join-semilattice. A subfamily $\Delta$ of $\mathcal{P}_{f}(L)$
is a hypercontact family if it satisfies the following conditions: 
\begin{enumerate}
\item If $\left\{ a_{1},\ldots,a_{n}\right\} \in\Delta$, then $a_{i}\neq0$,
for all $i\in\left\{ 1,\ldots,n\right\} $. 
\item If $U\subseteq V$ and $V\in\Delta$, then $U\in\Delta$. 
\item If $\left\{ a_{1},\ldots,a_{i},\ldots,a_{n}\right\} \in\Delta$ and
$a_{i}\leq b$, then $\left\{ a_{1},\ldots,a_{i},\ldots,a_{n},b\right\} \in\Delta$. 
\item If $b\neq0$, then $\left\{ b\right\} \in\Delta$. 
\end{enumerate}
A hypercontact semilattice (lattice) is a pair $\left<\mathbf{L},\Delta\right>$
such that $\mathbf{L}$ is a join-semilattice (lattice) and $\Delta$ is a
hypercontact family on $L$.

A hypercontact semilattice $\left<\mathbf{L},\Delta\right>$ is \emph{additive} if
$\left\{ a\vee b,c_{1},\ldots,c_{n}\right\} \in\Delta$ implies that
$\left\{ a,c_{1},\ldots,c_{n}\right\} \in\Delta$ or $\left\{ b,c_{1},\ldots,c_{n}\right\} \in\Delta$.

\end{definition}

Let $\left<L,\Delta\right>$ be a hypercontact semilattice. If $b\in L$,
$b\neq0$, and
\[b\leq b_{i},\; \text{ for all } i\in \{1,\ldots, n\},\]
by (4) and
(3) of Definition \ref{def:hypercontact semilattice}, we get  $\left\{ b,b_{1},\ldots,b_{n}\right\} \in\Delta$.
Moreover, by (2), it follows that $\left\{ b_{1},\ldots,b_{n}\right\} \in\Delta$.
Similarly, if $\left\{ a_{1},\ldots,a_{n}\right\} \in\Delta$ and
\[a_{i}\leq b_{i}\; \text{ for all } i\in \{1,\ldots, n\},\]
then, by (3) we obtain $\left\{ b_{1},\ldots,b_{n}\right\} \in\Delta$.

\begin{lemma}Let $\left<\mathbf{A},\mathrm{C}\right>$ be an $S$-algebra.
Then the family 
\[
\Delta_{\mathrm{C}}=\left\{ \left\{ a_{1},\ldots,a_{n}\right\} : 0\notin\mathrm{C}(a_{1},\ldots,a_{n}) \right\} 
\]
is an additive hypercontact family.

\end{lemma} 
\begin{proof}
If $\{a_{1},\ldots,a_{n}\} \in \Delta_{\mathrm{C}}$, then $0 \notin \mathrm{C}(a_{1},\ldots,a_{n})$. Since $\{a_{1},\ldots,a_{n}\} \subseteq \mathrm{C}(a_{1},\ldots,a_{n})$, it follows that $a_i \neq 0$ for all $i \in \{1,\ldots,n\}$, so (1) holds. 

To prove (2), let $\{a_{1},\ldots,a_{n}\} \subseteq \{b_{1},\ldots,b_{m}\}$ and $\{b_{1},\ldots,b_{m}\} \in \Delta_{\mathrm{C}}$. Then $0 \notin \mathrm{C}(b_{1},\ldots,b_{m})$. Since $\mathrm{C}$ is a closure operator, monotonicity implies that $\mathrm{C}(a_{1},\ldots,a_{n}) \subseteq \mathrm{C}(b_{1},\ldots,b_{m})$, and thus $0 \notin \mathrm{C}(a_{1},\ldots,a_{n})$. Therefore, $\{a_{1},\ldots,a_{n}\} \in \Delta_{\mathrm{C}}$.

For (3), let $\{a_{1},\ldots,a_{i},\ldots,a_{n}\} \in \Delta_{\mathrm{C}}$ and $a_{i} \leq b$. Thus, $0 \notin \mathrm{C}(a_{1},\ldots,a_{i},\ldots,a_{n})$ and $a_{i} \vee b = b$. It follows that
\[
\mathrm{C}(a_{1},\ldots,a_{n},b) = \mathrm{C}(a_{1},\ldots,a_{n},a_{i} \vee b) = \mathrm{C}(a_{1},\ldots,a_{n}) \cap \mathrm{C}(a_{1},\ldots,a_{n},b).
\]
Consequently, $\mathrm{C}(a_{1},\ldots,a_{n},b) \subseteq \mathrm{C}(a_{1},\ldots,a_{n})$. The reverse inclusion follows immediately from the monotonicity of $\mathrm{C}$. Thus, $\mathrm{C}(a_{1},\ldots,a_{n},b) = \mathrm{C}(a_{1},\ldots,a_{n})$, which implies that $0 \notin \mathrm{C}(a_{1},\ldots,a_{n},b)$, i.e., $\{a_{1},\ldots,a_{n},b\} \in \Delta_{\mathrm{C}}$.

For (4), if $b \neq 0$, then $0 \notin \mathrm{Fg}(b) = \mathrm{C}(b)$. Hence, $\{b\} \in \Delta_{\mathrm{C}}$. This completes the proof that $\Delta_{\mathrm{C}}$ is a hypercontact family. Finally, the additivity of $\Delta_{\mathrm{C}}$ is a direct consequence of the definition of $\Delta_{\mathrm{C}}$ and Remark \ref{additivity}.

\end{proof}

\section{Irreducible closed subsets}

In this section, we study irreducible closed subsets of Boolean closure spaces. We establish a separation lemma and show that these sets form a base for the closure system. We also characterize them as join-prime closed sets and relate them to Vakarelov's clans.
\medskip

We recall that a Boolean closure space is a closure space $\left<\mathbf{A},\mathrm{C}\right>$
where $\mathbf{A}$ is a Boolean algebra.

\begin{definition}\label{def: irreducible C-closed}
Let $\left<\mathbf{A},\mathrm{C}\right>$ be a Boolean
closure space. Let $F$ be a proper $\mathrm{C}$-closed subset. We say that $F$
is \emph{irreducible} if, for every $T_{1},T_{2}\in\mathcal{C}_{\mathrm{C}}$,
$T_{1}\cap T_{2}=F$ implies $F=T_{1}$ or $F=T_{2}$.

The poset of all irreducible $\mathrm{C}$-closed subsets is denoted by $X(\mathbf{A})$.
\end{definition}

The following lemma generalizes Lemma 10 from \cite{VakarelovSequentALg}.

\begin{lemma} \label{lem: separation}Let $\left<\mathbf{A},\mathrm{C}\right>$
be a finitary Boolean closure space, and suppose that $\mathrm{C}$
satisfies the property $\mathbf{Up}$ of Proposition \ref{prop: charact closure}.
Let $H\in\mathcal{C}_{\mathrm{C}}$, and $I\in\mathrm{Id}(\mathbf{A})$ such
that $H\cap I=\emptyset$. Then there exists an $F\in X(\mathbf{A})$ such that $H\subseteq F$ and $F\cap I=\emptyset$. \end{lemma} 
\begin{proof}
We apply Zorn's lemma. Consider the family 
\[
\mathcal{F}=\left\{ G\in\mathcal{C}_{\mathrm{C}}:H\subseteq G\text{ and }G\cap I=\emptyset\right\} .
\]
Since $H\in\mathcal{F},$ we have that $\mathcal{F}\neq\emptyset$. Observe that if $\Gamma$ is a non-empty chain of elements
of $\mathcal{F}$, then it is a nonempty upward-directed subfamily of $\mathcal{C}_{\mathrm{C}}$. From this, it is readily seen that $\bigcup\Gamma\in\mathcal{F}$. Thus, by Zorn's
lemma, there exists a maximal element $F$ of $\mathcal{F}$. We prove
that $F$ is irreducible. Let $T_{1}$, $T_{2}\in\mathcal{C}_{\mathrm{C}}$
such that $F=T_{1}\cap T_{2}$. Suppose that $F\neq T_{1}$ and $F\neq T_{2}$.
Because $H\subseteq F=T_{1}\cap T_{2}$ and $F$ is maximal in $\mathcal{F}$,
we have $T_{1}\cap I\neq\emptyset$ and $T_{2}\cap I\neq\emptyset$.
Then there exist $a,b\in A$ such that $a\in T_{1}\cap I$ and $b\in T_{2}\cap I$.
By $\mathbf{Up}$, every closed subset is an upset, thus $a\vee b\in T_{1}\cap T_{2}=F$, and since $I$ is an ideal, $a\vee b\in I$.
Then, $a\vee b\in F\cap I$, which is a contradiction. Therefore, $F$ must be irreducible. 
\end{proof}

\begin{corollary} \label{cor: closed is inter of irreducible} Let
$\left<\mathbf{A},\mathrm{C}\right>$ be a finitary Boolean closure space.
Suppose that $\mathrm{C}$ satisfies the property $\mathbf{Up}$ of
Proposition \ref{prop: charact closure}. Let $H\in\mathcal{C}_{\mathrm{C}}.$
If $a\notin H$, then there exists $F\in X(\mathbf{A})$ such
that $H\subseteq F$ and $a\notin F$. Thus, every proper $\mathrm{C}$-closed subset  is contained in an irreducible $\mathrm{C}$-closed subset. 
\end{corollary}

We recall that, given a closure system $\mathcal{C}$ on a set $A$,
a family $\mathcal{F}\subseteq\mathcal{P}(A)$ is said to be a base
of $\mathcal{C}$ when $\mathcal{C}$ is the smallest closure system
that contains $\mathcal{F}$. In such a case, $\mathcal{C}$ is referred
to as the closure system generated by $\mathcal{F}$. Then, as a consequence
of Corollary \ref{cor: closed is inter of irreducible} and Proposition
1.32 of \cite{F2016}, we obtain:

\begin{proposition}\label{clans for a basis} Let $\langle\mathbf{A},\mathrm{C}\rangle$
be a sequent algebra. Then, the set of irreducible closed sets of
$A$ is a base of $\mathcal{C}_{\mathrm{C}}$. Moreover,
the following hold: 
\begin{enumerate}
\item For every closed set $H$, if $H\neq A$, then 
\[
H=\bigcap\{F\in X(\mathbf{A})\colon H\subseteq F\},
\]
\item For every $Z\subseteq A$, either $\mathrm{C}(Z)=A$
or 
\[
\mathrm{C}(Z)=\bigcap\{F\in X(\mathbf{A})\colon Z\subseteq F\}.
\]
\end{enumerate}
\end{proposition}

\begin{lemma} \label{lem:charcat irreducible closed} Let $\left<\mathbf{A},\mathrm{C}\right>$
be a Boolean closure space satisfying the properties $\mathbf{Up}$
and $\mathbf{Jo}.$ Then 
\begin{enumerate}
\item A $\mathrm{C}$-closed subset $F$ is irreducible iff, for any $a,b\in A$, $a\vee b\in F$
implies $a\in F$ or $b\in F$. 
\item If $F$ is an irreducible $\mathrm{C}$-closed subset, then for each $a\in F$ there exists $U\in\Ul(\mathbf{A})$
such that $a\in U\subseteq F$. 
\end{enumerate}
\end{lemma} 
\begin{proof}
(1) $\Leftarrow)$ Suppose first that $F$ satisfies the stated join-prime
condition. Suppose that $T_{1},T_{2}\in\mathcal{C}_{\mathrm{C}}$
and $T_{1}\cap T_{2}=F$ but $T_{1}\nsubseteq F$ and $T_{2}\nsubseteq F$.
Then there are $a,b\in A$ such that $a\in T_{1}-F$ and $b\in T_{2}-F$.
By property $\mathbf{Up}$, every closed set is an upset. So, $a\vee b\in T_{1}\cap T_{2}=F$.
But this implies that $a\in F$ or $b\in F$, which is a contradiction.
Thus, $T_{1}=F$ or $T_{2}=F$.

$\Rightarrow)$ Now suppose that $F$ is irreducible. Let $a,b\in A$
such that $a\vee b\in F$. Since $F$ is closed, by property $\mathbf{Jo}$
we have $F=\mathrm{C}\left(F,a\vee b\right)=\mathrm{C}(F,a)\cap\mathrm{C}(F,b)$.
So, $\mathrm{C}(F,a)=F$ or $\mathrm{C}(F,b)=F$. Then, $a\in F$
or $b\in F$. \\
 (2) Let $a\in F$. Then $a\notin F^{c}$. Notice that by (1) and
$\mathbf{Up}$, it easily follows that $F^{c}$ is an ideal of $\mathbf{A}$.
So $\mathrm{Fg}(a)\cap F^{c}=\emptyset$ and by the Ultrafilter Theorem,
there exists $U\in\mathrm{Ul}(\mathbf{A})$ such that $a\in U$ and $U\cap F^{c}=\emptyset$.
Thus, $a\in U\subseteq F$, as desired. 
\end{proof}

\begin{lemma}\label{lem: X(A) and ultrafilters} Let $\left<\mathbf{A},\mathrm{C}\right>$
be a Boolean closure space satisfying $\mathbf{Fi}$, $\mathbf{Up}$,
and $\mathbf{Jo}$. Then 
\begin{enumerate}
\item $\mathrm{Fi}(\mathbf{A})\subseteq\mathcal{C}_{\mathrm{C}}$ and $\Ul(\mathbf{A})\subseteq X(\mathbf{A})$. 
\item If $F\in\mathrm{Fi}(\mathbf{A})$ and $G\in X(\mathbf{A})$ such that $F\subseteq G$,
then there exists $P\in\Ul(\mathbf{A})$ such that $F\subseteq P\subseteq G$. 
\end{enumerate}
\end{lemma} 
\begin{proof}
(1) Suppose that $\left<\mathbf{A},\mathrm{C}\right>$ satisfies $\mathbf{Fi}$
and let $F\in\mathrm{Fi}(\mathbf{A})$. If $a_{1},\ldots,a_{n}\in F$ and $a\in\mathrm{C}(a_{1},\ldots,a_{n})\subseteq\mathrm{Fg}(a_{1},\ldots,a_{n})$,
then, as $F$ is a filter, $a\in F$. \\
 (2) Notice that $F\cap G^{c}=\emptyset$. Since $G^{c}$ is an ideal
of $\mathbf{A}$, by the Ultrafilter Theorem, there exists $P\in\Ul(\mathbf{A})$ such
that $F\subseteq P$ and $P\cap G^{c}=\emptyset$. So $F\subseteq P\subseteq G$,
as claimed. 
\end{proof}

We conclude this section showing that the concept of a clan for \(S\)-algebras \cite{VakarelovSequentALg} and that of an irreducible $\mathrm{C}$-closed subset  for sequent algebras (Definition \ref{def: irreducible C-closed}) are merely two sides of the same coin, as summarized below.

\begin{remark}\label{S-clans} 
Let $\langle\mathbf{A},\vdash\rangle$ be an $S$-algebra. Recall from \cite{VakarelovSequentALg} that a subset $F \subseteq A$ is called an \emph{$S$-filter} provided that if $X \vdash b$, then $b \in F$ for any finite subset $X \subseteq F$. Moreover, an $S$-filter $F$ is a \emph{clan} if $0 \notin F$ and satisfies that $a \vee b \in F$ implies $a \in F$ or $b \in F$ for all $a,b \in A$.

Let $\langle\mathbf{A},\mathrm{C}_{\vdash}\rangle$ be its associated algebra via Proposition~\ref{prop: charact closure}. Since $\mathrm{C}_{\vdash}$ is a finitary closure operator, a subset $F$ of $A$ is $\mathrm{C}_{\vdash}$-closed iff $\mathrm{C}_{\vdash}(X) \subseteq F$ for any finite subset $X$ of $F$. Note that the latter condition is equivalent to stating that $X \vdash a$ entails $a \in F$ for any finite subset $X$ of $F$. Consequently, the notion of a $\mathrm{C}_{\vdash}$-closed subset coincides with that of an $S$-filter. Additionally, since property $\mathbf{Up}$ ensures that every $\mathrm{C}_{\vdash}$-closed subset is an upset of $\mathbf{A}$, it follows that $F$ is proper iff $0 \notin F$. Thus, this fact and Lemma~\ref{lem:charcat irreducible closed} lead us to conclude that the notions of $S$-clan and irreducible $\mathrm{C}_{\vdash}$-closed subset coincide.

\end{remark}

So, in the spirit of Remark \ref{S-clans}, for the rest of the paper,
we will refer to irreducible $\mathrm{C}$-closed sets of a sequent algebra as clans.

\section{Sequent spaces}

In this section, we introduce sequent spaces and study their basic properties. These structures provide the topological setting for the duality with sequent algebras developed in the next section.
\medskip

From now on, a sequent algebra will be represented as a pair $\left<\mathbf{A},\vdash\right>$
or as a pair $\left<\mathbf{A},\mathrm{C}\right>$, where $\mathrm{C}=\mathrm{C}_{\vdash}.$

\begin{lemma} \label{lem: maps alpa and beta}Let $\left<\mathbf{A},\mathrm{C}\right>$
be a sequent algebra. Let $\left<X(\mathbf{A}),\subseteq\right>$ be the poset
of all irreducible subsets of $\left<A,\mathrm{C}\right>$. Consider
the map $\beta_{\mathbf{A}}:A\rightarrow\mathcal{P}(X(\mathbf{A}))$ given by 
\[
\beta_{\mathbf{A}}(a)=\left\{ F\in X(\mathbf{A}):a\in F\right\} ,
\]
for each $a\in A$. Then, 
\begin{enumerate}
\item $\beta_{\mathbf{A}}$ is a join-preserving order-embedding. 
\item The family $\beta_{\mathbf{A}}\left[A\right]=\left\{ \beta_{\mathbf{A}}(a):a\in A\right\} $
is a basis of closed sets for a $T_{0}$ topology $\tau_{\beta_{\mathbf{A}}\left[A\right]}$. 
\item $\beta_{\mathbf{A}}(\neg a)=\cl(\beta_{\mathbf{A}}(a)^{c})=-\beta_{\mathbf{A}}(a),$ 
\item $\beta_{\mathbf{A}}(a)$ is a regular closed subset of the topological space
$\left<X(\mathbf{A}),\tau_{\beta_{\mathbf{A}}\left[A\right]}\right>$. 
\item $\mathcal{B}_{\mathbf{A}}=\left\langle \beta_{\mathbf{A}}[A],\cup,\overline{\cap},-,\emptyset,X(\mathbf{A})\right\rangle $
is a Boolean subalgebra of $\mathbf{RC}(X(\mathbf{A}))$, where $\overline{\cap}$
is defined by 
\[
\beta_{\mathbf{A}}(a)\,\overline{\cap}\,\beta_{\mathbf{A}}(b)=\cl(\In(\beta_{\mathbf{A}}(a)\cap\beta_{\mathbf{A}}(b))),
\]
for $a,b\in A$. 
\item \label{enu:closed}For any $X\cup\left\{ b\right\} \subseteq A$,
\[
b\in\mathrm{C}(X)\text{ iff }\bigcap\left\{ \beta_{\mathbf{A}}(a):a\in X\right\} \subseteq\beta_{\mathbf{A}}(b).
\]
\end{enumerate}
\end{lemma} 
\begin{proof}
(1) By Lemma \ref{lem:charcat irreducible closed} (1) we have $\beta_{\mathbf{A}}(a\vee b)=\beta_{\mathbf{A}}(a)\cup\beta_{\mathbf{A}}(b)$,
for any $a,b\in A$. It is clear that $\beta_{\mathbf{A}}(0)=\emptyset$, because
every irreducible set is proper. Let $a,b\in A$. If $a\leq b$ and
$a\in F\in X(\mathbf{A})$, then $b\in F$, because every closed set is an
upset. This proves that $\beta_{\mathbf{A}}(a)\subseteq\beta_{\mathbf{A}}(b)$. Conversely,
if $a\nleq b$, by the Ultrafilter Theorem there exists an ultrafilter $U$ of $\mathbf{A}$ such that
$a\in U$ and $b\notin U$. By Lemma \ref{lem: X(A) and ultrafilters} (1),
every ultrafilter is irreducible; hence $U\in\beta_{\mathbf{A}}(a)$ and $U\notin\beta_{\mathbf{A}}(b)$,
i.e., $\beta_{\mathbf{A}}(a)\nsubseteq\beta_{\mathbf{A}}(b)$.

(2) Observe that $\bigcap \{\beta_{\mathbf{A}}(a)\colon a\in A\}=\emptyset$ and by (1) $\beta_{\mathbf{A}}(a\vee b)=\beta_{\mathbf{A}}(a)\cup\beta_{\mathbf{A}}(b)$, thus $\beta[A]$ is a basis of closed sets for a topology $\tau_{\beta[A]}$. It is no hard to see that $\tau_{\beta[A]}$ is $T_0$. The proofs of (3) and (4) are similar to the proofs
given in Lemma 13 of \cite{VakarelovSequentALg}.

The proof of (5) follows from (3), (4) and the fact that $\beta_{\mathbf{A}}(a)\overline{\cap}\beta_{\mathbf{A}}(b)=\beta_{\mathbf{A}}(a\wedge b)$,
for all $a,b\in A$.

(6) Let $b\in\mathrm{C}(X)$. If $F\in\bigcap\left\{ \beta_{\mathbf{A}}(a):a\in X\right\} $,
then $a\in F$, for any $a\in X$. So, $X\subseteq F$ and $\mathrm{C}(X)\subseteq F$.
Thus, $b\in F$. We prove the converse by contraprosition. Suppose that $b\notin\mathrm{C}(X)$, i.e., $\mathrm{C}(X)\cap\downarrow b=\emptyset$.
By Lemma \ref{lem: separation} there exists $F\in X(\mathbf{A})$ such that
$\mathrm{C}(X)\subseteq F$ and $b\notin F$. Then, $F\in\bigcap\left\{ \beta_{\mathbf{A}}(a):a\in X\right\} $
and $F\notin\beta(b)$. Thus, $\bigcap\left\{ \beta_{\mathbf{A}}(a):a\in X\right\} \nsubseteq\beta_{\mathbf{A}}(b)$. 
\end{proof}

In what follows, a topological space $\left<X,\tau\right>$ with a
distinguished basis $\mathcal{B}$ of closed subsets of $\left<X,\tau\right>$
will be denoted by $\left<X,\tau_{\mathcal{B}}\right>$ or simply
by $\left<X,\mathcal{B}\right>$.

Recall that the relation $\vdash$ on $\mathbf{RC}(X)$ is defined
as: 
\[
U_{1},\ldots,U_{n}\vdash U\text{ iff }U_{1}\cap\ldots\cap U_{n}\subseteq U,
\]
for any $U_{1},\ldots,U_{n},U\in\mathrm{RC}(X)$. Moreover, the closure
operator 
\[\mathrm{C}_{\vdash}:\mathcal{P}(\mathrm{RC}(X))\rightarrow\mathcal{P}(\mathrm{RC}(X))\]
is defined as 
\[
U\in\mathrm{C}_{\vdash}(Y)\text{ iff }\bigcap\left\{ V:V\in Y\right\} \subseteq U
\]

The restriction of the relation $\vdash$ to $\mathcal{B}$ is 
\[
\vdash_{\mathcal{B}}{=}\vdash\cap\left(\mathcal{P}(\mathcal{B})\times\mathcal{B}\right).
\]

We now define the dual spaces of sequent algebras. The definition
is an adaptation of the notion of mereocompact space given by Robert
Goldblatt and Matt Grice in \cite{Goldblatt Mereocompactness}.

\begin{definition} \label{def:A-compingent-topological space}A sequent
space is a $T_{0}$ topological space $\left<X,\tau_{\mathcal{B}}\right>$
such that 
\begin{enumerate}
\item $\mathcal{B}$ is a subalgebra of the Boolean algebra $\mathbf{RC}(X)$
of regular closed subsets of $\left<X,\tau_{\mathcal{B}}\right>$. 
\item \label{condition compacity}For each irreducible closed subset $F$
of $\mathcal{B}$ and $\Delta\subseteq\mathcal{B}$, if $\bigcap F\subseteq\bigcup\Delta$,
then $F\cap\Delta\neq\emptyset$. 
\end{enumerate}
\end{definition}

It is readily seen that $\mathrm{C}_{\vdash_{\mathcal{B}}}$ is a
finitary closure operator satisfying $\mathbf{Fi}$, $\mathbf{One}$,
$\mathbf{Up}$, and $\mathbf{Jo}$. Therefore, if
$\left<X,\tau_{\mathcal{B}}\right>$ is a sequent space, we get that
$\langle\mathcal{B},\mathrm{C}_{\vdash_{\mathcal{B}}}\rangle$ is
a sequent algebra. 

\begin{lemma}\label{lem: Algebras to spaces}
Let $\left<\mathbf{A},\mathrm{C}\right>$
be a sequent algebra. Then, the topological space
$\left<X(\mathbf{A}),\tau_{\beta_{\mathbf{A}}\left[A\right]}\right>$
is a sequent space.
\end{lemma} 
\begin{proof}
    By Lemma \ref{lem: maps alpa and beta} (2) and (5),
$\left<X(\mathbf{A}),\tau_{\beta_{\mathbf{A}}\left[A\right]}\right>$ is a
$T_{0}$ space and that $\mathcal{B}_{\mathbf{A}}$ is a Boolean subalgebra
of $\mathbf{RC}(X(\mathbf{A}))$. Let $\Gamma$ be an irreducible closed subset of $\mathcal{B}_\mathbf{A}$ and
let $\Delta\subseteq\mathcal{B}_\mathbf{A}$ such that $\bigcap\Gamma\subseteq\bigcup\Delta$.
Let $F_{A}=\left\{ a\in A:\beta_\mathbf{A}(a)\in\Gamma\right\} $. We prove that
$F_{A}$ is an irreducible $\mathrm{C}$-closed subset of $A$. Let $a_{1},\ldots,a_{n}\in F_{A}$
and $b\in A$ such that $a_{1},\ldots,a_{n}\vdash b$. Then, by  Lemma \ref{lem: maps alpa and beta} (6) we have
that $\beta_{\mathbf{A}}(a_{1})\cap\ldots\cap\beta_{\mathbf{A}}(a_{n})\subseteq\beta_{\mathbf{A}}(b)$. Since
$\beta_{\mathbf{A}}(a_{1}),\ldots,\beta_{\mathbf{A}}(a_{n})\in\Gamma$ and $\Gamma$ is a closed
subset of $\left<\mathcal{B}_{\mathbf{A}},\vdash_{\mathcal{B}_{\mathbf{A}}}\right>$, we get
$\beta_{\mathbf{A}}(b)\in\Gamma$, i.e., $b\in F_{A}.$ Thus, $F_{A}$ is closed.
It is easy to see that if $a\vee b\in F_{A}$, then $a\in F_{A}$
or $b\in F_{A}.$ So, $F_{A}\in X(\mathbf{A})$. As $F_{A}\in\bigcap\Gamma\subseteq\bigcup\Delta$,
then there exists $\beta_{\mathbf{A}}(c)\in\Delta$ such that $F_{A}\in\beta_{\mathbf{A}}(c).$
So, $c\in F_{A}$ and by the definition of $F_{A}$,
we have $\beta_{\mathbf{A}}(c)\in\Gamma$. Thus, $\beta_{\mathbf{A}}(c)\in\Gamma\cap\Delta$ so $\Gamma\cap\Delta\neq \emptyset$. 
\end{proof}

\begin{lemma} Let $\left<\mathbf{A},\mathrm{C}\right>$ be a sequent algebra.
Then the sequent space $\left<X(\mathbf{A}),\tau_{\beta_{\mathbf{A}}\left[A\right]}\right>$
is compact. \end{lemma} 
\begin{proof}
Let $B\subseteq A$ be a non-empty subset such that $X(\mathbf{A})\subseteq\bigcup\left\{ \beta_\mathbf{A}(a)^{c}:a\in B\right\} $.
Suppose that for any finite subset $B_{0}$ of $B$ we have $X(\mathbf{A})\nsubseteq\bigcup\left\{ \beta_\mathbf{A}(a)^{c}:a\in B_{0}\right\} $.
Then the filter in the Boolean algebra $\mathcal{P}(X(\mathbf{A}))$ generated
by the family $\left\{ \beta_\mathbf{A}(a):a\in B\right\} $ is proper. Then
there exists an ultrafilter $P$ of $\mathcal{P}(X(\mathbf{A}))$ such that
$\left\{ \beta_\mathbf{A}(a):a\in B\right\} \subseteq P$. Consider the set $F_{P}=\left\{ a\in A:\beta_\mathbf{A}(a)\in P\right\} .$
We prove that $F_{P}\in X(A)$. Let $X=\left\{ a_{1},\ldots,a_{n}\right\} $
be a finite subset of $F_{P}$. Let $b\in\mathrm{C}(X)$. By Lemma
\ref{lem: maps alpa and beta}, $\beta_\mathbf{A}(a_{1})\cap\ldots\cap\beta_\mathbf{A}(a_{n})\subseteq\beta_\mathbf{A}(b)$.
As $\beta_\mathbf{A}(a_{1})\cap\ldots\cap\beta_\mathbf{A}(a_{n})\in P$, we have $\beta_\mathbf{A}(b)\in P$.
So, $b\in F_{P}$. Then $F_{P}$ is closed. Let $a\vee b\in F_{P}.$
So $\beta_\mathbf{A}(a\vee b)=\beta_\mathbf{A}(a)\cup\beta_\mathbf{A}(b)\in P$. Then, $\beta_\mathbf{A}(a)\in P$
or $\beta_\mathbf{A}(b)\in P$, i.e., $a\in F_{P}$ or $b\in F_{P}$. It is clear
that $a\in F_{p}$ for any $\beta(a)\in P$. Thus, $B\subseteq F_{p}$,
i.e., $F_{p}\notin\bigcup\left\{ \beta_\mathbf{A}(a)^{c}:a\in B\right\} $, which
is a contradiction. Thus, $\left<X(A),\tau_{\beta_{\mathbf{A}}[A]}\right>$ is
compact. 
\end{proof}

\section{Duality theory}

In this section, we introduce the categories of sequent algebras and sequent spaces and establish a duality between them. We first characterize the appropriate morphisms and then construct canonical representations for both algebras and spaces. These representations yield a dual equivalence between the corresponding categories, which restricts to a dual equivalence between the categories of sequent algebras with sequent homomorphisms and sequent spaces with sequent morphisms.

\subsection{Categories of sequent algebras and sequent spaces}

Here and below, a join-homomorphism is understood to preserve
$0$ and binary joins.

\begin{lemma}\label{lem:chacact homo contact}Let $\left<\mathbf{A},\mathrm{C}\right>$
and $\left<\mathbf{B},\mathrm{D}\right>$ be sequent algebras. Let $g:A\rightarrow B$
be a monotonic map. Then the following conditions are equivalent: 
\begin{enumerate}
\item $g$ is a join-homomorphism such that $g(1)=1$, and $g[\mathrm{C}(X)]\subseteq\mathrm{D}(g[X]),$
for each $X\subseteq A$. 
\item $g^{-1}[F]\in X(\mathbf{A})$, for each $F\in X(\mathbf{B})$. 
\end{enumerate}
\end{lemma} 
\begin{proof}
$(1)\Rightarrow(2):$ Let $F\in X(\mathbf{B})$. We start by showing that $g^{-1}[F]$ is a $\mathrm{C}$-closed subset of $A$. To this end, we will prove that $\mathrm{C}(X)\subseteq g^{-1}[F]$, for every $X\subseteq g^{-1}[F]$. Let $a\in\mathrm{C}(X)$ and $X\subseteq g^{-1}[F]$. Then, $g(a)\in g[\mathrm{C}(X)]$.
Since $g[X]\subseteq F$ and $F$ is $\mathrm{D}$-closed, by (1) we have $g[C(X)]\subseteq \mathrm{D}(g[X])\subseteq F$.
This yields $a\in g^{-1}[F]$, as desired. 

Since $0=g(0)\notin F$, then it follows $0\notin g^{-1}[F]$, so $g^{-1}[F]$ is proper. Finally, to prove that $g^{-1}[F]$ is irreducible, let $a\vee b\in g^{-1}[F]$. Then $g(a\vee b)=g(a)\vee g(b)\in F$. Since $F$ is irreducible, Lemma \ref{lem:charcat irreducible closed} (1) yields $g(a)\in F$ or $g(b)\in F$. Hence, $a\in g^{-1}[F]$ or $b\in g^{-1}[F]$. Therefore, by the same result as above, $g^{-1}[F]\in X(\mathbf{A})$.

$(2)\Rightarrow(1):$ As $g$ is monotone, we have $g(a),g(b)\leq g(a\vee b)$ for all $a,b\in A$. Suppose that there exist $a,b\in A$ such that $g(a\vee b)\nleq g(a)\vee g(b)$. Then, by the Ultrafilter Theorem, there exists an ultrafilter $P$ of $\mathbf{B}$ such that $g(a\vee b)\in P$ and $g(a)\vee g(b)\notin P$. Consequently, $a\vee b\in g^{-1}[P]$, whereas $a,b\notin g^{-1}[P]$. This contradicts the irreducibility of $g^{-1}[P]$, which follows from Lemma \ref{lem: X(A) and ultrafilters} (1) and (2). 

To show that \(g(0)=0\), suppose that \(g(0)\nleq 0\). By the Ultrafilter Theorem, there exists an ultrafilter \(Q\) of \(\mathbf{B}\) such that \(g(0)\in Q\) and \(0\notin Q\). Thus, \(0\in g^{-1}[Q]\). By Lemma \ref{lem: X(A) and ultrafilters} (1), every ultrafilter of \(\mathbf{B}\) is an irreducible \(\mathrm{C}\)-closed subset of \(B\). Hence, by (2), \(g^{-1}[Q]\) is irreducible and, in particular, proper. This yields a contradiction. Therefore, \(g(0)=0\), and consequently, \(g\) is a join-homomorphism.

Now we show that \(g(1)=1\). Suppose, towards a contradiction, that \(g(1)<1\). Then \(0<\neg g(1)\), so \(\neg g(1)\nleq 0\). By the Ultrafilter Theorem, there exists an ultrafilter \(P\) of \(\mathbf{B}\) containing \(\neg g(1)\). Since \(g\) is monotone, \(g(a)\leq g(1)\) for every \(a\in A\). Hence, \(g(a)\notin P\) for every \(a\in A\), and thus \(g^{-1}[P]=\emptyset\). This contradicts (2), since \(\emptyset\) is not \(\mathrm{C}\)-closed: indeed, \(\mathbf{One}\) implies that \(1\in\mathrm{C}_{A}(\emptyset)\). Therefore, \(g(1)=1\).

Finally, let $X\subseteq A$. We show that
$g[\mathrm{C}(X)]\subseteq\mathrm{D}(g[X])$ by contradiction.
Suppose that there exists $a\in\mathrm{C}(X)$ such that
$g(a)\notin\mathrm{D}(g[X])$. Since $\mathrm{D}(g[X])$ is
$\mathrm{D}$-closed, Corollary \ref{cor: closed is inter of irreducible}
yields an $F\in X(\mathbf{B})$ such that
$\mathrm{D}(g[X])\subseteq F$ and $g(a)\notin F$.
Since $\mathrm{D}$ is a closure operator, we have $g[X]\subseteq F$,
and hence $X\subseteq g^{-1}[F]$. Since $\mathrm{C}$ is a closure
operator, it follows that
$\mathrm{C}(X)\subseteq\mathrm{C}(g^{-1}[F])$.
By (2), $g^{-1}[F]$ is $\mathrm{C}$-closed, so
$\mathrm{C}(X)\subseteq g^{-1}[F]$. In particular,
$a\in g^{-1}[F]$, contradicting $g(a)\notin F$.
\end{proof}

\begin{definition}Let $\left<\mathbf{A},\mathrm{C}\right>$ and $\left<\mathbf{B},\mathrm{D}\right>$
be sequent algebras. We say that a monotone map $g:A\rightarrow B$ is a sequent
semi-homomorphism if satisfies any
of the equivalent conditions of Lemma \ref{lem:chacact homo contact}.
\end{definition}

\begin{corollary}\label{points as morphisms} For every sequent algebra
$\left<\mathbf{A},\mathrm{C}\right>$, there is a bijective correspondence
between $X(\mathbf{A})$ and the set of sequent semi-homomorphisms
from $\left<\mathbf{A},\mathrm{C}\right>$ to $\left<\mathbf{0},\mathsf{Fg}^{\mathbf{0}}\right>$,
where $\mathbf{0}$ stands for the two-element Boolean algebra. \end{corollary} 
\begin{proof}
First, observe that the only irreducible closed subset of
$\left<\mathbf{0},\mathsf{Fg}^{\mathbf{0}}\right>$ is $\{1\}$,
which is also the only ultrafilter of $\mathbf{0}$. Let
$h:\left<\mathbf{A},\mathrm{C}\right>\to
\left<\mathbf{0},\mathsf{Fg}^{\mathbf{0}}\right>$
be a sequent semi-homomorphism. By Lemma
\ref{lem:chacact homo contact}, we have
$h^{-1}[\{1\}]\in X(\mathbf{A})$.

Conversely, let $F\in X(\mathbf{A})$, and let
$\chi_F:A\to\{0,1\}$ be the characteristic function of $F$.
Clearly, $\chi_F$ is monotone and
$\chi_F^{-1}[\{1\}]=F$. Hence, by Lemma
\ref{lem:chacact homo contact}, $\chi_F$ is a sequent
semi-homomorphism from
$\left<\mathbf{A},\mathrm{C}\right>$ to
$\left<\mathbf{0},\mathsf{Fg}^{\mathbf{0}}\right>$.

Therefore, the assignments
$h\mapsto h^{-1}[\{1\}]$ and $F\mapsto\chi_F$
establish the desired bijection.
\end{proof}

It is easy to see that the composition of two sequent semi-homomorphisms
is a sequent semi-homomorphism, and that the identity function on
a sequent algebra is a sequent semi-homomorphism. Thus we can define
the category $\mathsf{SA}$ whose objects are sequent algebras
and whose morphisms are sequent semi-homomorphisms. We write $\mathsf{SAB}$ for the subcategory of $\mathsf{SA}$, whose objects are sequent algebras and whose morphisms are the Boolean homomorphisms between them. In such a case the morphisms of $\mathsf{SAB}$ will be refered to as \emph{sequent homomorphisms}.
\medskip 

Let $\left<\mathbf{A},\vdash_{A}\right>$ and $\left<\mathbf{B},\vdash_{B}\right>$ be sequent algebras, and write $\mathrm{C}_{A}=\mathrm{C}_{\vdash_{A}}$ and $\mathrm{C}_{B}=\mathrm{C}_{\vdash_{B}}$. Observe that a join-homomorphism $h\colon A\rightarrow B$ preserving $1$ is a sequent semi-homomorphism if, for any $a_{1},\ldots,a_{n},b\in A$,
\[
a_{1},\ldots,a_{n}\vdash_{A}b\ \text{ implies }\ h(a_{1}),\ldots,h(a_{n})\vdash_{B}h(b).
\]
Equivalently,
\[
h(a_{1}),\ldots,h(a_{n})\nvdash_{B}h(b)\ \text{ implies }\ a_{1},\ldots,a_{n}\nvdash_{A}b.
\]

\begin{remark}\label{rem: sequent contact}
Recall that if $\left<\mathbf{A},\vdash_{A}\right>$ is a sequent
algebra, then we can define a contact relation $\mathsf{C}\subseteq A\times A$
as follows: 
\[
\left(a,b\right)\in\mathsf{C}\text{ iff }a,b\nvdash0.
\]
Thus, if $h:\left<\mathbf{A},\vdash_{A}\right>\rightarrow \left<\mathbf{B},\vdash_{B}\right>$ is a sequent semi-homomorphism, then $h$
is a contact homomorphism \cite{C2026}, because if 
\[
h(a),h(b)\nvdash_{B}0=h(0),
\]
then 
\[
a,b\nvdash_{A}0.
\]

\end{remark}

\begin{definition}
Let $\left<X_{1},\mathcal{B}_{1}\right>$ and $\left<X_{2},\mathcal{B}_{2}\right>$
be two sequent spaces. A function $f:X_{1}\rightarrow X_{2}$ is a
sequent map if $f^{-1}[U]\in \mathcal{B}_{1}$, for each $U\in \mathcal{B}_{2}$. We say
that $f$ is a \emph{sequent semi-morphism} if $f$ is a sequent map
and the function $f^{*}:\mathcal{B}_{2}\rightarrow \mathcal{B}_{1}$ given by $f^{*}(U)=f^{-1}[U]$
is a join-homomorphism that preserves top elements. Moreover, we say that $f$ is a \emph{sequent
morphism} if,  $f^{*}$ is a Boolean homomorphism between
$\mathcal{B}_{2}$ and $\mathcal{B}_{1}$.\end{definition}

Notice that every sequent semi-morphism, as well as every sequent morphism is continuous because the pre-images of closed
sets are closed sets.

\begin{remark}\label{rem: SeqSpm is not full}
Altough every sequent morphism is a sequent semi-morphism, the converse does not hold in general. To see this, consider $\mathbb{R}$ with its usual topology and the
sequent space $\langle\mathbb{R},\mathbf{RC}(\mathbb{R})\rangle$.
Let $\theta:\mathbb{R}\to\mathbb{R}$ be the constant map defined by
$\theta(x)=0$ for every $x\in\mathbb{R}$. For every
$U\in\mathrm{RC}(\mathbb{R})$, we have
\[
\theta^{\ast}(U)=
\begin{cases}
\mathbb{R} & \text{if }0\in U,\\
\emptyset & \text{if }0\notin U.
\end{cases}
\]
Hence, $\theta$ is a sequent map. Moreover,
$\theta^{\ast}(\emptyset)=\emptyset$ and
$\theta^{\ast}(\mathbb{R})=\mathbb{R}$. Furthermore, for every
$U,V\in\mathrm{RC}(\mathbb{R})$,
\[
\theta^{\ast}(U)\cup\theta^{\ast}(V)=
\begin{cases}
\mathbb{R} & \text{if }0\in U\cup V,\\
\emptyset & \text{if }0\notin U\cup V.
\end{cases}
\]
Thus,
\[
\theta^{\ast}(U)\cup\theta^{\ast}(V)
=\theta^{\ast}(U\cup V),
\]
so $\theta^{\ast}$ is a join-homomorphism. Nevertheless,
as shown in \cite{Goldblatt Mereocompactness},
$\theta^{\ast}$ preserves neither complementation nor binary meets
in $\mathrm{RC}(\mathbb{R})$. Therefore, $\theta$ is a sequent
semi-morphism that is not a sequent morphism.
\end{remark}

Let $\left<X_{1},\mathcal{B}_{1}\right>$,
$\left<X_{2},\mathcal{B}_{2}\right>$, and
$\left<X_{3},\mathcal{B}_{3}\right>$ be sequent spaces.
If $f:X_{1}\to X_{2}$ and $g:X_{2}\to X_{3}$ are sequent maps,
then
\[
f^{*}\circ g^{*}=(g\circ f)^{*}.
\]
Since $f$ and $g$ are sequent maps, the induced maps
$f^{*}$ and $g^{*}$ are join - homomorphisms that preserve top elements. Consequently,
$(g\circ f)^{*}$ is also a sequent semi-homomorphism.

Furthermore, for every sequent space $\left<X,\mathcal{B}\right>$,
the identity map $\mathrm{Id}_{X}:X\to X$ is a sequent map, and
$(\mathrm{Id}_{X})^{*}:\mathcal{B}\to\mathcal{B}$ is a sequent
semi-homomorphism. Thus, sequent spaces and sequent maps form a
category, denoted by $\mathsf{SeqSp}$. We denote by $\mathsf{SeqSp}_{\mathrm{m}}$ the subcategory of $\mathsf{SeqSp}$ whose morphisms are the sequent morphisms.  Observe that Remark \ref{rem: SeqSpm is not full} shows that the inclusion $\mathsf{SeqSp}_{\mathrm{m}}\hookrightarrow \mathsf{SeqSp}$ is not full.

\begin{proposition}\label{ssh to seq maps} Let $h:\langle\mathbf{A},\mathrm{C}\rangle\to\langle\mathbf{B},\mathrm{D}\rangle$
be a sequent semi-homomorphism. Then $h^{-1}:X(\mathbf{B})\to X(\mathbf{A})$ is a sequent semi-morphism between the sequent spaces $\langle X(\mathbf{B}),\mathcal{B}_{\mathbf{B}}\rangle$ and $\langle X(\mathbf{A}),\mathcal{B}_{\mathbf{A}}\rangle$.

\end{proposition} 
\begin{proof}
From Lemma \ref{lem:chacact homo contact} (2) it is clear that $h^{-1}$
maps clans to clans. For the last part, notice that
if $\nu:=h^{-1}$, then, the following diagram 
\begin{equation}
\xymatrix{\mathcal{P}(X(\mathbf{A}))\ar[r]^-{\nu^{-1}} & \mathcal{P}(X(\mathbf{B}))\\
\mathbf{A}\ar@{_{(}->}[u]^-{\beta_{\mathbf{A}}}\ar[r]_-{h} & \mathbf{B}\ar@{_{(}->}[u]_-{\beta_{\mathbf{B}}}
}
\label{commutative diagram}
\end{equation}
commutes. So, from the latter, the fact that $h$ is a sequent
semi-homomorphism and Lemma \ref{lem: maps alpa and beta} (1), it is not hard to see that for every $a,b\in A$:
\[
\nu^{-1}(\beta_{\mathbf{A}}(a)\cup\beta_{\mathbf{A}}(b))=\nu^{-1}(\beta_{\mathbf{A}}(a))\cup\nu^{-1}(\beta_{\mathbf{A}}(b)).
\]
Finally, from Lemma \ref{props SequentALg}, it follows that $\beta_{\mathbf{A}}(1)=X(\mathbf{A})$
and $\beta_{\mathbf{B}}(1)=X(\mathbf{B})$. Thus $\nu^{-1}(X(\mathbf{A}))=X(\mathbf{B})$.
Similarly, since $\beta_{\mathbf{A}}(0)=\emptyset=\beta_{\mathbf{B}}(0)$,
one gets $\nu^{-1}(\emptyset)=\emptyset$. This shows that $h^{-1}$ is a a sequent semi-homomorphism. 
\end{proof}

\begin{proposition}\label{seq maps to ssh} Let $\left<X_{1},\mathcal{B}_{1}\right>$
and $\left<X_{2},\mathcal{B}_{2}\right>$ be two sequent spaces. Let $f:X_{1}\rightarrow X_{2}$
be a sequent semi-morphism. Then the function $f^{*}:\mathcal{B}_{2}\rightarrow \mathcal{B}_{1}$
is a sequent semi-homomorphism between the sequent algebras $\left<\mathcal{B}_{2},\mathrm{C}_{\vdash_{B_{2}}}\right>$
and $\left<\mathcal{B}_{1},\mathrm{C}_{\vdash_{B_{1}}}\right>$.

\end{proposition} 
\begin{proof}
It is clear that $f^{*}$ is a join-homomorphism that preserves the top element between the Boolean
algebras $\mathcal{B}_{2}$ and $\mathcal{B}_{1}$. We prove that $f^{*}$ satisfies condition
(1) of Lemma \ref{lem:chacact homo contact}. Let $U_{1},\ldots,U_{n},U\in \mathcal{B}_{2}$.
Suppose that $U_{1},\ldots,U_{n}\vdash_{\mathcal{B}_{2}}U$, i.e.,
$U_{1}\cap\ldots\cap U_{n}\subseteq U$. Then $f^{-1}[U_{1}]\cap\ldots\cap f^{-1}[U_{n}]\subseteq f^{-1}[U]$.
Thus, $f^{*}[U_{1}],\ldots,f^{*}[U_{n}]\vdash_{\mathcal{B}_{1}}f^{*}[U]$. 
\end{proof}

\subsection{The duality theorem}
In this subsection, we characterize isomorphisms in the categories under consideration and construct canonical representations of sequent algebras and sequent spaces. We also establish the naturality of these representations.
\medskip

 Now, we characterize isomorphisms in $\mathsf{SA}$.

\begin{proposition}\label{isos of sequent algebras} A sequent semi-homomorphism
$h:\langle\mathbf{A},\mathrm{C}\rangle\to\langle\mathbf{B},\mathrm{D}\rangle$
is an isomorphism if and only if $h$ is bijective and $h[\mathrm{C}(Z)]=\mathrm{D}(h[Z])$,
for every $Z\in\mathcal{P}(A)$. \end{proposition} 
\begin{proof}
On the one hand, let us assume that $h$ is an isomorphism in $\mathsf{SA}$.
Then, it is clear that the existence of the inverse of $h$ in $\mathsf{SA}$,
say $g$, makes $h$ bijective. Now take $Z\in\mathcal{P}(A)$. Since
$h$ is a sequent semi-homomorphism by assumption, in order to prove our
claim, it only remains to prove that $\mathrm{D}(h[Z])\subseteq h[\mathrm{C}(Z)]$.
Indeed, since $h[Z]\in\mathcal{P}(B)$ and $g$ is a sequent semi-homomorphism,
we have 
\[
g[\mathrm{D}(h[Z])]\subseteq\mathrm{C}(g[h[Z]])=\mathrm{C}(Z).
\]
From this fact, it is readily seen that the claim holds.
On the other hand, let us assume that $h$ is bijective and $h[\mathrm{C}(Z)]=\mathrm{D}(h[Z])$.
We only need to check that the inverse map of $h$
is a sequent semi-homomorphism. It is easy to see that that $h^{-1}$ is a join-homomorphism that preserves $1$. It remains to show that for every $W\in\mathcal{P}(B)$,
$(h^{-1})[\mathrm{D}(W)]\subseteq\mathrm{C}((h^{-1})[W])$.
Certainly, since $W=h[(h^{-1})[W]]$, then, by hypothesis, we get
\[
\mathrm{D}(W)=\mathrm{D}(h[(h^{-1})[W]])\subseteq h[\mathrm{C}((h^{-1})[W])].
\]
 The bijectivity of $h^{-1}$ clearly allows us to conclude that $(h^{-1})[\mathrm{D}(W)]\subseteq\mathrm{C}((h^{-1})[W])$, thus, $h^{-1}$ is a sequent semi-homomorphism, as desired.
\end{proof}

Observe that Proposition \ref{isos of sequent algebras} allows us
to say that an isomorphism of sequent algebras is precisely a Bounded join-semilattice isomorphism  $h:\mathbf{A}\rightarrow\mathbf{B}$ such that 
\[
h(a_{1}),\ldots,h(a_{n})\vdash_{B}h(b)\;\text{ iff }\;a_{1},\ldots,a_{n}\vdash_{A}b.
\]

Let $\left<\mathbf{A},\mathrm{C}\right>$ be a sequent algebra. Recall that by Lemma \ref{lem: Algebras to spaces}, 
$\left<X(\mathbf{A}),\tau_{\beta_{\mathbf{A}}\left[A\right]}\right>$ is a sequent space, so by Lemma \ref{lem: maps alpa and beta}, $\left<\mathcal{B}_\mathbf{A},\mathrm{C}_{\vdash_{\mathcal{B}_\mathbf{A}}}\right>$ is again a sequent algebra.

\begin{lemma} \label{unit duality} The  map $\beta_\mathbf{A}:A\rightarrow\beta_\mathbf{A}\left[A\right]$
is a sequent algebra isomorphism between the sequent algebras $\left<\mathbf{A},\mathrm{C}\right>$ and $\left<\mathcal{B}_\mathbf{A},\mathrm{C}_{\vdash_{\mathcal{B}_\mathbf{A}}}\right>$. Moreover, this isomorphism is natural.
\end{lemma} 
\begin{proof}
Combining Lemma \ref{lem: maps alpa and beta} (1) and Proposition \ref{isos of sequent algebras} yield that $\beta_{\mathbf{A}}$ is an
isomorphism of sequent algebras. Now, we prove naturality. Let $h:\left<\mathbf{A},\mathrm{C}_{\mathbf{A}}\right>
\longrightarrow
\left<\mathbf{B},\mathrm{C}_{\mathbf{B}}\right>$ be a sequent semi-homomorphism, and let $\widehat{h}:X(\mathbf{B})\longrightarrow X(\mathbf{A})$ with $\widehat{h}(F)=h^{-1}[F]$, be the induced map. By Proposition~\ref{ssh to seq maps},
$\widehat{h}$ is a sequent semi-morphism and by diagram \eqref{commutative diagram} its induced map $\widehat{h}^{\,*}:
\mathcal{B}_{\mathbf{A}}\longrightarrow
\mathcal{B}_{\mathbf{B}}$ satisfies, for every $a\in A$,
\[
\begin{aligned}
\widehat{h}^{\,*}(\beta_{\mathbf{A}}(a))
&=\widehat{h}^{-1}[\beta_{\mathbf{A}}(a)]\\
&=\{F\in X(\mathbf{B}):a\in h^{-1}[F]\}\\
&=\{F\in X(\mathbf{B}):h(a)\in F\}\\
&=\beta_{\mathbf{B}}(h(a)).
\end{aligned}
\]
Thus,
\[
\widehat{h}^{\,*}\circ\beta_{\mathbf{A}}
=\beta_{\mathbf{B}}\circ h,
\]
which proves the naturality of $\beta$.
\end{proof}

The following result is straightforward. It provides a characterization of isomorphisms in $\mathsf{SeqSp}$.

\begin{proposition}\label{prop: isos SeqSp} 
A sequent map $f:X_{1}\rightarrow X_{2}$ is
an isomorphism in $\mathsf{SeqSp}$ iff the following hold: 
\begin{enumerate}
\item $f$ is bijective. 
\item For each $U\in\mathcal{B}_{1}$, $f\left[U\right]=\left\{ f(x):x\in U\right\} \in\mathcal{B}_{2}$. 
\end{enumerate}
\end{proposition}

Let $\left<X,\mathcal{B}\right>$ be a sequent space. We consider the function
$\varepsilon_{X}:X\rightarrow X(\mathcal{B})$ given by 
\[
\varepsilon_{X}(x)=\left\{ U\in \mathcal{B}:x\in U\right\} .
\]
For each $x\in X$, the set $\varepsilon_{X}(x)$ is an upset, $\emptyset\notin\varepsilon_{X}(x)$,
and if $U\cup V\in\varepsilon_{X}(x)$, then $U\in\varepsilon_{X}(x)$
or $V\in\varepsilon_{X}(x)$. Moreover, if $U_{1},\ldots,U_{n}\in\varepsilon_{X}(x)$
and $U_{1},\ldots,U_{n}\vdash_{B}U$, then $U\in\varepsilon_{X}(x)$.
Hence $\varepsilon_{X}(x)$ is a clan of the sequent algebra $\langle \mathcal{B}, \mathrm{C}_{\vdash_{\mathcal{B}}}\rangle$.

\begin{proposition}\label{counit duality}
Let $\left<X,\mathcal{B}\right>$ be a sequent space, and let
$\left<\mathcal{B},\mathrm{C}_{\vdash_{\mathcal{B}}}\right>$
be its associated sequent algebra. Then
$\left<X,\mathcal{B}\right>$ and
$\left<X(\mathcal{B}),\mathcal{B}_{\mathcal{B}_X}\right>$
are isomorphic in $\mathsf{SeqSp}$, where $\mathcal{B}_X$ is
the subalgebra of $\mathbf{RC}(X(\mathcal{B}))$ whose universe is
$\beta_{\mathcal{B}}[\mathcal{B}]$. Moreover, this isomorphism
is natural.
\end{proposition}
\begin{proof}
We first prove that $\varepsilon_{X}$ is bijective. Injectivity follows
from the $T_{0}$ property and the fact that $\mathcal{B}$ is a basis of closed
sets: if $x\neq y$, some $U\in \mathcal{B}$ contains exactly one of $x,y$,
so $\varepsilon_{X}(x)\neq\varepsilon_{X}(y)$. For surjectivity, let $F\in X(\mathcal{B})$ be a clan and put $\Delta=B\setminus F$.
Since $F\cap\Delta=\emptyset$, condition~\ref{condition compacity}
in Definition \ref{def:A-compingent-topological space} implies 
\[
\bigcap F\nsubseteq\bigcup\Delta.
\]
Choose $x\in\bigcap F\setminus\bigcup\Delta$. Then, for every $U\in \mathcal{B}$,
$x\in U$ if and only if $U\in F$. Consequently $\varepsilon_{X}(x)=F$. Further, observe that for every $U\in \mathcal{B}$, 
\[
\varepsilon^{-1}_{X}[\beta_{\mathcal{B}}(U)]=U.
\]
Since $\varepsilon_{X}$ is bijective, this also gives $\varepsilon_{X}[U]=\beta_{\mathcal{B}}(U)$.
Thus $\varepsilon_{X}$ and its inverse preserve the specified bases, so
Proposition \ref{prop: isos SeqSp} 
yields the claim.

To prove naturality, let
$f:X_1\to X_2$ be a sequent semi-morphism between the sequent
spaces $\langle X_1,\mathcal{B}_1\rangle$ and
$\langle X_2,\mathcal{B}_2\rangle$. Let $f^*:\mathcal{B}_2\to\mathcal{B}_1$ defined as $f^*(U)=f^{-1}[U]$ 
be the induced sequent semi-homomorphism, and define
\[
\widehat{f}:X(\mathcal{B}_1)\to X(\mathcal{B}_2),
\qquad
\widehat{f}(F)=(f^*)^{-1}[F].
\]
By Proposition \ref{ssh to seq maps}, $\widehat{f}$ is a
sequent semi-morphism.

Observe that for every $x\in X_1$,
\[
\begin{aligned}
\widehat{f}(\varepsilon_{X_1}(x))
&=(f^*)^{-1}[\varepsilon_{X_1}(x)]\\
&=\{U\in\mathcal{B}_2:f^*(U)\in\varepsilon_{X_1}(x)\}\\
&=\{U\in\mathcal{B}_2:x\in f^{-1}[U]\}\\
&=\{U\in\mathcal{B}_2:f(x)\in U\}\\
&=\varepsilon_{X_2}(f(x)).
\end{aligned}
\]
Hence,
$\widehat{f}\circ\varepsilon_{X_1}
=\varepsilon_{X_2}\circ f$, which proves the naturality of
the isomorphism.
\end{proof}

Let $h:\langle\mathbf{A},\mathrm{C}\rangle\to\langle\mathbf{B},\mathrm{D}\rangle$
be a sequent semi-homomorphism. By Lemma \ref{lem: Algebras to spaces} and Proposition \ref{ssh to seq maps},
it is not hard to see that the assignments $F(\langle\mathbf{A},\mathrm{C}_{\mathbf{A}}\rangle)=\langle X(\mathbf{A}),\mathcal{B}_{\mathbf{A}}\rangle$
and $F(h)=h^{-1}$ determine a functor $F:\mathsf{SA}^{\mathsf{op}}\to\mathsf{SeqSp}$.
Conversely, if $f:\langle X_{1},\mathcal{B}_{1}\rangle\to\langle X_{2},\mathcal{B}_{2}\rangle$
is a sequent semi-morphism, by Proposition \ref{seq maps to ssh},
it follows that the assignments $G(\langle X_{1},\mathcal{B}_{1}\rangle)=\langle\mathcal{B}_{1},\mathrm{C}_{\mathcal{B}_{1}}\rangle$
and $G(f)=f^{-1}$ also determine a functor $G:\mathsf{SeqSp}^{\mathsf{op}}\to\mathsf{SA}$.
Hence, by Proposition \ref{counit duality}
it follows that $F$ and $G$ constitute a categorical equivalence.
So, we have proved:

\begin{theorem}\label{dually category} The categories $\mathsf{SA}$
and $\mathsf{SeqSp}$ are dually equivalent. Moreover, this duality restricts to a dual equivalence between the categories $\mathsf{SAB}$ and $\mathsf{SeqSp}_{\mathrm{m}}$.
\end{theorem}
\begin{proof}
For the moreover part, let $h:\langle\mathbf{A},\mathrm{C}\rangle\to\langle\mathbf{B},\mathrm{D}\rangle$
be a sequent semi-homomorphism and let $\widehat{h}^{\,*}:
\mathcal{B}_{\mathbf{A}}\longrightarrow
\mathcal{B}_{\mathbf{B}}$ be the induced map of the semi-morphism $\widehat{h}:X(\mathbf{B})\longrightarrow X(\mathbf{A})$ defined as $\widehat{h}(F)=h^{-1}[F]$. We show that, $h$ preserves negation if and
only if $\widehat{h}^{\,*}$ does. Consequently, $h$ is a Boolean homomorphism if and only if, $\widehat{h}^{\,*}$ does.

Indeed, suppose that for every $a\in A$, $h(\neg a)=\neg h(a)$. From the
commutativity of diagram (\ref{commutative diagram}), the hypothesis
and Lemma \ref{lem: maps alpa and beta} (3) it follows: 
\[
\begin{array}{rcl}
-\widehat{h}^{\,*}(\beta_{\mathbf{A}}(a)) & = & -\beta_{\mathbf{B}}(h(a))\\
 & = & \beta_{\mathbf{B}}(\neg h(a))\\
 & = & \beta_{\mathbf{B}}(h(\neg a))\\
 & = & \widehat{h}^{\,*}(\beta_{\mathbf{A}}(\neg a))\\
 & = & \widehat{h}^{\,*}(-\beta_{\mathbf{A}}(a)).
\end{array}
\]
Conversely, suppose that $-\widehat{h}^{\,*}(\beta_{\mathbf{A}}(a))=\widehat{h}^{\,*}(-\beta_{\mathbf{A}}(a))$,
for every $a\in A$. Then, again by the commutativity of diagram (\ref{commutative diagram})
and Lemma \ref{lem: maps alpa and beta} (1) it is immediate that $h(\neg a)=\neg h(a)$. 

The preceding argument shows that $F$ and $G$ restrict to functors between $\mathsf{SAB}^{\mathsf{op}}$ and $\mathsf{SeqSp}_{\mathrm{m}}$. These restricted functors are faithful, because $F$ and $G$ are, and they are full, since by the equivalence $h$ is a Boolean homomorphism if and only if $\widehat{h}^{\,*}$ is, so that a morphism of $\mathsf{SA}$ (resp.\ $\mathsf{SeqSp}$) lies in $\mathsf{SAB}$ (resp.\ $\mathsf{SeqSp}_{\mathrm{m}}$) exactly when its image does. Moreover, every object of $\mathsf{SA}$ belongs to $\mathsf{SAB}$ and every object of $\mathsf{SeqSp}$ belongs to $\mathsf{SeqSp}_{\mathrm{m}}$, and the natural isomorphisms $\beta_{\mathbf{A}}$ and $\varepsilon_X$ of Lemma~\ref{unit duality} and Proposition~\ref{counit duality} lie in the subcategories, because an isomorphism of Boolean algebras with respect to the order is a Boolean isomorphism. Hence the restricted functors are essentially surjective, and so they establish a dual equivalence between $\mathsf{SAB}$ and $\mathsf{SeqSp}_{\mathrm{m}}$.
\end{proof}

\section{Positive filters and sequent congruences}

In this section, we study sequent congruences and their connections with filters and the dual topological spaces. We first introduce positive filters and show that they coincide with the regular filters determined by a natural subordination associated with the sequent relation. We then establish an isomorphism between the lattice of sequent congruences and the lattice of regular filters, and characterize these congruences dually in terms of minimally decreasing closed subsets of the associated sequent space.

\begin{proposition}\label{closed-closed} Let $\langle\mathbf{A},\mathrm{C}\rangle$
be a sequent algebra. Then the following hold: 
\begin{enumerate}
\item If $Y$ is a closed set of $\langle X(\mathbf{A}),\tau_{\beta_{\mathbf{A}}[A]}\rangle$,
then \[F_{Y}=\{a\in A\colon Y\subseteq\beta_\mathbf{A}(a)\}\]
is a $\mathrm{C}$-closed set of $A$.
\item If $F$ is a $\mathrm{C}$-closed set of $A$, then 
\[
Y_{F}=\bigcap\{\beta_{\mathbf{A}}(a)\colon a\in F\}=\{G\in X(\mathbf{A})\colon F\subseteq G\}
\]
is a closed set of $\langle X(\mathbf{A}),\tau_{\beta_{\mathbf{A}}[A]}\rangle$. 
\end{enumerate}
Furthermore, the poset of closed sets of $\langle X(\mathbf{A}),\tau_{\beta_{\mathbf{A}}[A]}\rangle$
is dually isomorphic to the poset $\langle\mathcal{C}_{\mathrm{C}},\subseteq\rangle$.
\end{proposition} 
\begin{proof}
(1) We only prove that $\mathrm{C}(F_{Y})\subseteq F_{Y}$,
since the converse inclusion follows from the fact that $\mathrm{C}$
is a closure operator on $A$. To this end, let $x\in\mathrm{C}(F_{Y})$.
As $\mathrm{C}$ is finitary, there exist $a_{1},\ldots,a_{n}\in F_{Y}$
such that $x\in\mathrm{C}(a_{1},\ldots,a_{n})$. By Lemma
\ref{lem: maps alpa and beta} (6),
\[
\beta_{\mathbf{A}}(a_{1})\cap\ldots\cap\beta_{\mathbf{A}}(a_{n})\subseteq\beta_{\mathbf{A}}(x).
\]
By assumption, $Y\subseteq\beta_{\mathbf{A}}(a_{i})$ for every $1\leq i\leq n$,
and thus from the above we get $Y\subseteq\beta_{\mathbf{A}}(x)$. Hence, $x\in F_{Y}$,
as required.

(2) This follows immediately since, for every $a\in A$, $\beta_{\mathbf{A}}(a)$
is a closed set by Lemma \ref{lem: maps alpa and beta}, and arbitrary
intersections of closed sets are closed.

For the ``furthermore'' part, we first note that, since $\beta_{\mathbf{A}}$
preserves finite joins, $\beta_{\mathbf{A}}(0)=\emptyset$ and $\beta_{\mathbf{A}}(1)=X(\mathbf{A})$
(because $1\in\mathrm{C}(\emptyset)$ belongs to every
closed set), the family $\beta_{\mathbf{A}}[A]$ is closed under finite unions.
Hence $\beta_{\mathbf{A}}[A]$ is a base (not only a subbase) of closed sets, and
every closed set of $X(\mathbf{A})$ is an intersection of sets of
the form $\beta_{\mathbf{A}}(a)$. Let $Y$ be a closed set of the dual sequent
space of $\langle\mathbf{A},\mathrm{C}\rangle$. It is
clear that $Y\subseteq Y_{F_{Y}}$. For the converse inclusion, suppose
that there is $P\in Y_{F_{Y}}$ such that $P\notin Y$. By the preceding
observation, there exists $a\in A$ with $Y\subseteq\beta_{\mathbf{A}}(a)$ and
$P\notin\beta_{\mathbf{A}}(a)$. Thus $a\in F_{Y}$. Since, by assumption, $F_{Y}\subseteq P$,
we get $a\in P$, which is a contradiction. Therefore $Y_{F_{Y}}=Y$.

On the other hand, let $F$ be a closed set of $\mathrm{C}$.
Observe that $a\in F_{Y_{F}}$ if and only if $\bigcap\{\beta_{\mathbf{A}}(b)\colon b\in F\}\subseteq\beta_{\mathbf{A}}(a)$.
By Lemma \ref{lem: maps alpa and beta} (6), the latter is equivalent
to $a\in\mathrm{C}(F)=F$. Hence $F_{Y_{F}}=F$. Therefore
the assignments $Y\mapsto F_{Y}$ and $F\mapsto Y_{F}$ are inverses
of each other, and both are clearly antitone. 
\end{proof}

Let $\left<X,\tau_{\mathcal{B}}\right>$ be a sequent space. We define
the relation $\sqsubseteq$ on $X$ as the dual of the specialization
order, that is, 
\[
x\sqsubseteq y\text{ iff }y\in\mathrm{cl}(\{x\}).
\]
Since $X$ is $T_{0}$, $\sqsubseteq$ is a partial order. For $x\in X$
we write $(x]=\{z\in X\colon z\sqsubseteq x\}$, and $\mathrm{min}X$
denotes the set of minimal elements of $\langle X,\sqsubseteq\rangle$.

Now let $\langle\mathbf{A},\mathrm{C}\rangle$ be a sequent
algebra and consider its dual sequent space $\langle X(\mathbf{A}),\tau_{\beta_{\mathbf{A}}[A]}\rangle$.
Since $\beta_{\mathbf{A}}[A]$ is a base of closed sets, for every $F\in X(\mathbf{A})$
we have 
\[
\mathrm{cl}(\{F\})=\bigcap\{\beta_{\mathbf{A}}(a)\colon F\in\beta_{\mathbf{A}}(a)\}=\{G\in X(\mathbf{A})\colon F\subseteq G\}.
\]
Hence $F\sqsubseteq G$ iff $F\subseteq G$, and there is no ambiguity
in writing $\langle X(\mathbf{A}),\sqsubseteq\rangle$ or $\langle X(\mathbf{A}),\subseteq\rangle$. 

\begin{proposition}\label{maximals are ultrafilters} Let $\langle\mathbf{A},\mathrm{C}\rangle$
be a sequent algebra. Then 
\[
\mathrm{min}X(\mathbf{A})=\Ul(\mathbf{A}).
\]
\end{proposition} 
\begin{proof}
Let $P\in\Ul(\mathbf{A})$. By Lemma \ref{lem: X(A) and ultrafilters} (1),
$P\in X(\mathbf{A})$. Suppose that $F\in X(\mathbf{A})$ and $F\subseteq P$.
By \textbf{One} $1\in F$, thus by Lemma \ref{lem:charcat irreducible closed} (2)
there is $Q\in\Ul(\mathbf{A})$ such that $Q\subseteq F$. Then $Q\subseteq F\subseteq P$,
and the maximality of $Q$ yields $Q=F=P$. Thus $P\in\mathrm{min}X(\mathbf{A})$.
Conversely, let $F\in\mathrm{min}X(\mathbf{A})$. Again by Lemma \ref{lem:charcat irreducible closed} (2),
there exists $Q\in\Ul(\mathbf{A})$ such that $Q\subseteq F$. Since
$Q\in X(\mathbf{A})$, the minimality of $F$ gives $Q=F$, so $F\in\Ul(\mathbf{A})$. 
\end{proof}

\begin{definition}\label{mdecrasing closed set} Let $\left<X,\tau_{\mathcal{B}}\right>$
be a sequent space. We say that a closed set $Y$ of $X$ is \emph{minimally
decreasing} (m-decreasing, for short) if $\mathrm{min}X\cap(x]\subseteq Y$,
for every $x\in Y$, where $(x]$ and $\mathrm{min}X$
are taken with respect to $\sqsubseteq$. \end{definition}

The set of all m-decreasing closed sets of $X$ will be denoted by
$\mathrm{MD}(X)$.

\begin{definition}\label{positive filter} Let $\langle\mathbf{A},\mathrm{C}\rangle$
be a sequent algebra. A filter $F$ of $\mathbf{A}$ is \emph{positive}
if for every $a\in F$ and $G\in X(\mathbf{A})$, if $F\subseteq G$,
then $\neg a\notin G$. \end{definition}

The set of all positive filters of $\mathbf{A}$ will be denoted by
$\mathrm{PF}(\mathbf{A})$. Recall that, by $\mathbf{Fi}$, every
filter $F$ of $\mathbf{A}$ is closed:
$\mathrm{C}(F)\subseteq\mathrm{Fg}(F)=F$
(Lemma \ref{lem: X(A) and ultrafilters}(1)). The converse is false
in general: a closed set need not be a filter. Indeed, consider the
four-element Boolean algebra
$\mathbf{A}=\{0,a,b,1\}$, where $a$ and $b$ are distinct atoms,
and equip it with the closure operator
\[
\mathrm{C}_{\min}(Z)=
\begin{cases}
\{1\}, & Z=\emptyset,\\
\uparrow Z, & Z\neq\emptyset.
\end{cases}
\]
In Proposition \ref{fibre-complete}, we will show that $\langle \mathbf{A}, \mathrm{C}_{\min}\rangle$ is a sequent algebra. Then $Z=\{a,b,1\}$ is closed, since
$\mathrm{C}_{\min}(Z)=\uparrow Z=Z$, but it is not a filter,
as $a,b\in Z$ while $a\wedge b=0\notin Z$.

\begin{theorem}\label{postive to mdecreceasing} Let $\langle\mathbf{A},\mathrm{C}\rangle$
be a sequent algebra. Then the following hold: 
\begin{enumerate}
\item If $F$ is a positive filter of $\mathbf{A}$, then $Y_{F}$ is an
m-decreasing closed set of $\langle X(\mathbf{A}),\tau_{\mathcal{B}_{\mathbf{A}}}\rangle$. 
\item If $Y$ is an m-decreasing closed set of $\langle X(\mathbf{A}),\tau_{\mathcal{B}_{\mathbf{A}}}\rangle$,
then $F_{Y}$ is a positive filter of $\mathbf{A}$. 
\end{enumerate}
Furthermore, the posets $\langle\mathrm{MD}(X(\mathbf{A})),\subseteq\rangle$
and $\langle\mathrm{PF}(\mathbf{A}),\subseteq\rangle$ are dually
isomorphic. \end{theorem} 
\begin{proof}
(1) Let $F$ be a positive filter. Since $F$ is closed, $Y_{F}$
is closed by Proposition \ref{closed-closed}. Let $G\in Y_{F}$ and
let $U\in\mathrm{min}X(\mathbf{A})\cap(G]$. By Proposition \ref{maximals are ultrafilters},
$U$ is an ultrafilter and $U\subseteq G$. We show that $F\subseteq U$.
Suppose, to the contrary, that $a\in F\setminus U$. Since $U$ is
an ultrafilter, $\neg a\in U\subseteq G$. But $F\subseteq G$, contradicting
the positivity of $F$. Hence $U\in Y_{F}$, and therefore $Y_{F}$
is m-decreasing.

(2) Let $Y$ be an m-decreasing closed subset of $\langle X(\mathbf{A}),\tau_{\mathcal{B}_{\mathbf{A}}}\rangle$.
By Proposition \ref{closed-closed}, $F_{Y}$ is closed with respect
to $\mathrm{C}$. Since a closed set need not be a filter,
we verify directly that $F_{Y}$ is a filter. Clearly $1\in F_{Y}$,
and $F_{Y}$ is an upset. Let $a,b\in F_{Y}$ and $P\in Y$. By Lemma
\ref{lem:charcat irreducible closed} (2), there exists an ultrafilter
$U$ of $\mathbf{A}$ such that $U\subseteq P$. By Proposition \ref{maximals are ultrafilters},
$U\in\mathrm{min}X(\mathbf{A})\cap(P]$, and since $Y$ is m-decreasing,
$U\in Y$. Thus $a,b\in U$, so $a\wedge b\in U\subseteq P$. As $P\in Y$
was arbitrary, $Y\subseteq\beta(a\wedge b)$, and hence $a\wedge b\in F_{Y}$.
Therefore $F_{Y}$ is a filter.

It remains to prove positivity. Suppose that $a\in F_{Y}$, $G\in X(\mathbf{A})$,
$F_{Y}\subseteq G$, and $\neg a\in G$. By Lemma~\ref{lem:charcat irreducible closed} (2),
there is an ultrafilter $U$ of $\mathbf{A}$ with $\neg a\in U\subseteq G$. Since
$G\in Y_{F_{Y}}=Y$ by Proposition~\ref{closed-closed}, and $Y$
is m-decreasing, we have $U\in Y$. But $a\in F_{Y}$ means $Y\subseteq\beta(a)$,
so $a\in U$, contradicting $\neg a\in U$. Hence $F_{Y}$ is positive.

The final statement follows from (1), (2), and Proposition~\ref{closed-closed},
since every positive filter is closed, so that $F_{Y_{F}}=F$ for
each $F\in\mathrm{PF}(\mathbf{A})$. 
\end{proof}

\begin{proposition}\label{postive simpler} Let $\langle\mathbf{A},\mathrm{C}\rangle$
be a sequent algebra and let $F$ be a filter of $\mathbf{A}$. Then
$F$ is positive if and only if, for every $a\in F$, 
\[
\mathrm{C}\bigl(F\cup\{\neg a\}\bigr)=A.
\]
\end{proposition} 
\begin{proof}
Suppose first that $F$ is positive and that, for some $a\in F$,
$\mathrm{C}(F\cup\{\neg a\})\neq A$. By Corollary \ref{cor: closed is inter of irreducible},
there exists a clan $G\in X(\mathbf{A})$ containing $F\cup\{\neg a\}$.
This contradicts the positivity of $F$.

Conversely, assume that $\mathrm{C}(F\cup\{\neg a\})=A$
for every $a\in F$. If $F$ were not positive, there would exist
$a\in F$ and $G\in X(\mathbf{A})$ such that $F\subseteq G$ and
$\neg a\in G$. Since $G$ is closed, 
\[
A=\mathrm{C}(F\cup\{\neg a\})\subseteq G,
\]
contradicting the fact that a clan is proper. 
\end{proof}

We now introduce the notion of sequent congruence. A sequent congruence
is a Boolean congruence that is compatible with the closure operator,
or equivalently with the sequent relation of a sequent algebra. This
notion extends the concept of $\prec$-congruence, also known as quasi-modal
congruence, introduced and studied in \cite{Cel:q-m} and \cite{Celani-SubdirectlyQUasi}. We formalize this concept in Definition \ref{def:sequent congruences}, but first we require to establish some technical facts.
\\

Recall that Proposition \ref{prop: charact closure} establishes an equivalence between S-algebras (Definition \ref{def: S-algebras}) and sequent algebras. Henceforth, for the rest of this section we work with the presentation of sequent algebras as S-algebras and write them as $\langle A,\vdash\rangle$.

\begin{lemma} Let $\left<A,\vdash\right>$ be a sequent algebra.
For each finite subset $\left\{ a,c\right\} \cup B\subseteq A$, the
set 
\[
I\left(a,B,c\right)=\left\{ x\in A:a,x,B\vdash c\right\} 
\]
is an ideal of $A$. \end{lemma} 
\begin{proof}
Since $0\vdash c$ by condition (4) of the definition of $S$-algebra,
condition (2) gives $a,0,B\vdash c$, so $0\in I(a,B,c)$. Let $x,y\in A$
such that $x\leq y$ and $y\in I\left(a,B,c\right)$. By Lemma \ref{props SequentALg}(3),
$a,x,B\vdash c$. If $x,y\in I\left(a,B,c\right)$, then $a,x,B\vdash c$
and $a,y,B\vdash c$. By condition (6) of the definition of $S$-algebra,
$a,x\vee y,B\vdash c$. Thus, $I\left(a,B,c\right)$ is an ideal of
$A$. 
\end{proof}

\begin{remark} Let $\left<A,\vdash\right>$ be a sequent algebra.
If $B=\emptyset$ and $c=0$, then 
\[
A\setminus I(a,\emptyset,0)=\left\{ x\in A:a,x\nvdash0\right\} =\textnormal{\ensuremath{\left\{  x\in A:\left(a,x\right)\in\mathrm{C}_{\vdash}\right\} } ,}
\]
where $\mathrm{C}_{\vdash}$is the contact relation associated with
the sequent relation $\vdash$ established in Remark \ref{rem: sequent contact}. 
\end{remark}

\begin{definition}\label{def:sequent congruences}Let $\left<A,\vdash\right>$
be a sequent algebra. Let $\theta$ be a Boolean congruence. We say
that $\theta$ is a sequent congruence if 
\begin{enumerate}
\item[$\mathrm{(SCon)}$] For each $\left(a,b\right)\in\theta$ and for each $\left\{ c\right\} \cup B\subseteq_{f}A$,
if $d\in I(\neg a,B,c)$, then there exists $e\in I(\neg b,B,c)$
such that $(d,e)\in\theta$. 
\end{enumerate}
\end{definition}

The condition $\mathrm{(SCon)}$ together with the symmetry of $\theta$
implies that if $\left(a,b\right)\in\theta$ and $e\in I(\neg b,B,c)$,
then there exists $d\in I(\neg a,B,c)$ such that $(d,e)\in\theta$.

\begin{remark}
For each finite subset $\left\{ c\right\} \cup B$
of $A$, we denote by $\prec^{c}_{B}$ the binary relation on $A$
defined by 
\[
a\prec^{c}_{B}b\text{ iff }a,\neg b,B\vdash c.
\]
If $B=\emptyset$, we write $a\prec^{c}b$ instead of $a\prec^{c}_{\emptyset}b$.
Finally, we define 
\[
\prec\coloneqq\bigcap_{\{c\}\cup B\subseteq_{f}A}\prec^{c}_{B}.
\]
\end{remark}

\begin{lemma}\label{lem: prec simple} Let $\left<A,\vdash\right>$
be a sequent algebra. For all $a,b\in A$, 
\[
a\prec b\text{ iff }a\prec^{0}b\text{ iff }a,\neg b\vdash0.
\]
\end{lemma} 
\begin{proof}
If $a\prec b$, then in particular $a\prec^{0}_{\emptyset}b$. Conversely,
suppose $a,\neg b\vdash0$ and let $\left\{ c\right\} \cup B\subseteq_{f}A$.
By condition (2) of the definition of $S$-algebra, $a,\neg b,B\vdash0$,
and since $0\leq c$, Lemma \ref{props SequentALg} (4) gives $a,\neg b,B\vdash c$.
Thus $a\prec^{c}_{B}b$. 
\end{proof}

Let $\mathbf{A}$ be a Boolean algebra and let $\prec$ be a binary
relation on $A$. Following \cite{C2026}, we say that $\prec$ is a
\emph{subordination} on $A$ if it satisfies the following conditions:
\begin{enumerate}
    \item[\textnormal{(S1)}] $0 \prec 0$ and $1 \prec 1$.
    \item[\textnormal{(S2)}] $a \prec b$ and $a \prec c$ imply
    $a \prec b \wedge c$.
    \item[\textnormal{(S3)}] $a \prec c$ and $b \prec c$ imply
    $a \vee b \prec c$.
    \item[\textnormal{(S4)}] $a \leq b \prec c \leq d$ imply
    $a \prec d$.
\end{enumerate}
\begin{lemma}\label{lem: prec subordination}Let $\mathbf{A}$ be
a sequent algebra. Then $\prec$ is a subordination on $A$. \end{lemma} 
\begin{proof}
We use Lemma \ref{lem: prec simple}. Since $0\vdash0$, condition
(2) gives $0,1\vdash0$ and $1,0\vdash0$, that is, $0\prec0$ and
$1\prec1$. If $a\prec b$ and $a\prec c$, then $a,\neg b\vdash0$
and $a,\neg c\vdash0$, so by condition (6), $a,\neg b\vee\neg c\vdash0$,
i.e., $a\prec b\wedge c$. Similarly, if $a\prec c$ and $b\prec c$,
condition (6) gives $a\vee b,\neg c\vdash0$, i.e., $a\vee b\prec c$.
Finally, if $a\leq b\prec c\leq d$, then $\neg d\leq\neg c$ and
two applications of Lemma \ref{props SequentALg} (3) give $a,\neg d\vdash0$,
i.e., $a\prec d$. 
\end{proof}

\begin{proposition}\label{prop: positive iff regular} Let $\left<A,\vdash\right>$
be a sequent algebra and let $F$ be a filter of $A$. Then $F$ is
positive if and only if for each $a\in F$ there exists $f\in F$
such that $f\prec a$. \end{proposition} 
\begin{proof}
Let $a\in F$. Since closed sets are upsets (by $\mathbf{Up}$), $\mathrm{C}(F\cup\{\neg a\})=A$
iff $0\in\mathrm{C}(F\cup\{\neg a\})$. As $\mathrm{C}$ is finitary
and, by condition (2), we may always add $\neg a$ and an element
of $F$ to the premises. This holds iff there exist $f_{1},\ldots,f_{n}\in F$,
with $n\geq1$, such that $f_{1},\ldots,f_{n},\neg a\vdash0$. By
Lemma \ref{props SequentALg} (1) and (3), this is equivalent to $f,\neg a\vdash0$
for $f=f_{1}\wedge\ldots\wedge f_{n}$, which belongs to $F$ because
$F$ is a filter. By Lemma~\ref{lem: prec simple}, $f,\neg a\vdash0$
means $f\prec a$. The result now follows from Proposition \ref{postive simpler}. 
\end{proof}

At the light of Proposition \ref{prop: positive iff regular}, positive
filters can be described purely in terms of the subordination $\prec$.
We shall also call them \emph{regular} filters \cite{Celani-SubdirectlyQUasi}. 

\begin{definition}\label{def:regular filter}Let $\left<A,\vdash\right>$
be a sequent algebra. A filter $F$ of $A$ is \emph{regular} if for
each $f\in F$ there exists $z\in F$ such that $z\prec f$,
i.e., $z,\neg f\vdash0$. By Proposition \ref{prop: positive iff regular},
a filter is regular if and only if it is positive. \end{definition}

\begin{theorem}\label{thm: congruences regular}Let $\left<A,\vdash\right>$
be a sequent algebra. Then the following hold: 
\begin{enumerate}
\item If $F$ is a regular filter of $\mathbf{A}$, then 
\[\theta(F)=\left\{ \left(a,b\right)\in A^{2}:\exists f\in F\,\left(a\wedge f=b\wedge f\right)\right\} \]
is a sequent congruence. 
\item If $\theta$ is a sequent congruence, then $1_{\theta}=\left\{ a\in A:\left(a,1\right)\in\theta\right\} $
is a regular filter. 
\end{enumerate}
Furthermore, the poset of sequent congruences of $\mathbf{A}$ is
isomorphic to the poset of regular filters of $\mathbf{A}$. \end{theorem} 
\begin{proof}
(1) Let $F$ be a regular filter. It is well known that $\theta(F)$
is a Boolean congruence. Let $(a,b)\in\theta(F)$, let $\left\{ c\right\} \cup B\subseteq_{f}A$
be fixed, and let $e\in I(\neg a,B,c)$, i.e., 
\[
e,\neg a,B\vdash c.
\]
We need to find $d\in A$ such that $d,\neg b,B\vdash c$ and $(e,d)\in\theta(F)$.
Since $(a,b)\in\theta(F)$, there exists $f\in F$ such that $a\wedge f=b\wedge f$.
As $F$ is regular, there exists $z\in F$ such that $z,\neg f\vdash0$;
by Lemma \ref{lem: prec simple}, $z,\neg f,B\vdash c$. By Lemma
\ref{props SequentALg} (3) we have $e\wedge z,\neg a,B\vdash c$ and
$e\wedge z,\neg f,B\vdash c$, so condition (6) gives 
\[
e\wedge z,\neg a\vee\neg f,B\vdash c.
\]
Since 
\[
\neg a\vee\neg f=\neg(a\wedge f)=\neg(b\wedge f)=\neg b\vee\neg f
\]
and $\neg b\leq\neg b\vee\neg f$, Lemma \ref{props SequentALg} (3)
yields 
\[
e\wedge z,\neg b,B\vdash c.
\]
Since $z\in F$ and $(e\wedge z)\wedge z=e\wedge z$, we get $\left(e,e\wedge z\right)\in\theta(F)$.
Take $d=e\wedge z$. Then $d,\neg b,B\vdash c$ and $(e,d)\in\theta(F)$.

(2) Let $\theta$ be a sequent congruence. It is well known that $1_{\theta}$
is a filter. Let $f\in1_{\theta}$, i.e., $(1,f)\in\theta$. We apply
$\mathrm{(SCon)}$ with $c=0$ and $B=\emptyset$. Since $\neg1=0$
and $0,1\vdash0$, we have $1\in I(\neg1,\emptyset,0)$. Hence there
exists $z\in I(\neg f,\emptyset,0)$ such that $(1,z)\in\theta$,
that is, $z,\neg f\vdash0$ and $z\in1_{\theta}$. By Lemma \ref{lem: prec simple},
$z\prec f$. Thus $1_{\theta}$ is a regular filter.

For the last part, recall that the assignments $\theta\mapsto1_{\theta}$
and $F\mapsto\theta(F)$ are mutually inverse order-preserving bijections
between Boolean congruences and filters of $A$, i.e., $\theta(1_{\theta})=\theta$
and $1_{\theta(F)}=F$. By (1) and (2), they restrict to mutually
inverse order-isomorphisms between sequent congruences and regular
filters. 
\end{proof}

Combining Theorem \ref{postive to mdecreceasing}, Proposition \ref{prop: positive iff regular}
and Theorem \ref{thm: congruences regular}, we obtain:

\begin{corollary}\label{cor: congruences md} Let $\mathbf{A}=\left<A,\vdash\right>$
be a sequent algebra and let $F$ be a filter of $A$. Then $F$ is
regular if and only if $Y_{F}=\left\{ G\in X(\mathbf{A}):F\subseteq G\right\} $
is an m-decreasing closed set of $\langle X(\mathbf{A}),\tau_{\mathcal{B}_{\mathbf{A}}}\rangle$.
Consequently, the poset of sequent congruences of $\mathbf{A}$ is
isomorphic to $\langle\mathrm{PF}(\mathbf{A}),\subseteq\rangle$ and
dually isomorphic to $\langle\mathrm{MD}(X(\mathbf{A})),\subseteq\rangle$.
\end{corollary} 
\begin{proof}
If $F$ is regular, it is positive by Proposition \ref{prop: positive iff regular},
so $Y_{F}$ is m-decreasing by Theorem \ref{postive to mdecreceasing} (1).
Conversely, if $Y_{F}$ is m-decreasing, then $F_{Y_{F}}$ is positive
by Theorem \ref{postive to mdecreceasing} (2). Since $F$ is closed
(by $\mathbf{Fi}$), $F_{Y_{F}}=F$ by Proposition \ref{closed-closed},
so $F$ is positive, hence regular. The last statement follows from
Theorems \ref{postive to mdecreceasing} and \ref{thm: congruences regular}. 
\end{proof}

\section{$\mathsf{SAB}$ is topological over $\mathsf{BA}$}

Given a topological space $\langle X,\tau\rangle$ and a subset $A\subseteq X$,
the subspace topology on $A$ is simply the smallest topology that
still makes the inclusion map $A\hookrightarrow X$ continuous. Likewise,
if we have a family of topological spaces $\{X_{i}\}_{i\in I}$, the
product topology on $\prod_{i\in I}X_{i}$ is the coarsest topology
that makes all projection maps $p_{i}$ continuous.

Now, generalizing this idea, suppose we have a family of spaces $\{\langle X_{i},\tau_{i}\rangle\}_{i\in I}$,
a set $X$, and a collection of functions $G=\{f_{i}:X\to X_{i}\}_{i\in I}$.
Our goal is to find the ‘best’ topology on $X$---the one with the
fewest open sets---such that every function $f_{i}$ is continuous.
This leads to what is called the \emph{initial topology}, which is
the topology $\tau$ generated by the subbasis: 
\[
\mathcal{F}_{G}=\{f^{-1}_{i}[V_{i}]\mid V_{i}\in\tau_{i},i\in I\}.
\]
In other words, it is the coarsest topology that
makes all the functions in $G$ continuous. It is well known that
the initial topology has the following property:

\begin{proposition}\label{structured - initial lift} Let $\langle X,\tau_{G}\rangle$
be a topological space where $\tau_{G}$ is the initial topology induced
by the family of functions $G=\{f_{i}:X\to X_{i}\colon i\in I\}$.
Let $\langle Y,\tau_{Y}\rangle$ be a topological space. A function
$g:Y\to X$ is continuous if and only if the composition $f_{i}\circ g:Y\to X\to X_{i}$
is continuous for each $i\in I$. \end{proposition}

Moreover, this property characterizes the initial topology.

\begin{proposition}\label{uniqueness - initial lift} Let $X$ be
a set, and let $G=\{f_{i}:\langle X,\tau\rangle\to\langle X_{i},\tau_{i}\rangle\}_{i\in I}$
be a collection of continuous functions. A topology $\tau$ on $X$
is equal to the initial topology $\tau_{G}$ if and only if $\tau$
satisfies the conditions of Proposition \ref{structured - initial lift}.
\end{proposition}

Let $\mathsf{Top}$ be the category of topological spaces and continuous
functions and $\mathsf{Set}$ be the category of sets and functions.
If $O:\mathsf{Top}\to\mathsf{Set}$ denotes the forgetful functor,
that is, the one defined by the assignments $O(\langle X,\tau\rangle)=X$
and $O(\langle X,\tau\rangle\xrightarrow{f}\langle Y,\mu\rangle)=X\xrightarrow{f}Y$,
then Propositions \ref{structured - initial lift} and \ref{uniqueness - initial lift}
may be summarized as follows: For every family of functions $G=\{f_{i}:X\to O(\langle X_{i},\tau_{i}\rangle)\colon i\in I\}$
there exists a unique topology $\tau$ on $X$ such that the family
$\{f_{i}:\langle X,\tau\rangle\to\langle X_{i},\tau_{i}\rangle\colon i\in I\}$
is a family of continuous functions satisfying Proposition \ref{structured - initial lift}.
\\

We emphasize that the initial topology problem is a particular case
of a much more general situation. In fact, categories in which this
problem has a unique solution have a specific name: they are called
\emph{topological categories}, and they have been well studied in
the literature (see \cite{AHS2006} and the references therein). Their
importance lies in the fact that they provide a categorical approach
that captures, in categorical terms, the essence of topology. \\

\medskip{}
 This section is devoted to showing that the category $\mathsf{SAB}$
is topological over $\mathsf{BA}$. To achieve this, we first recall
some definitions, all of which can be found in Chapter VI of \cite{AHS2006}.
\\

Consider two categories $\mathsf{A}$ and $\mathsf{X}$, and a functor
$U:\mathsf{A}\to\mathsf{X}$. A pair $(\mathsf{A},U)$ is said to
be a \emph{concrete category over} $\mathsf{X}$ if $U$ is faithful.
A morphism $g:U(B)\to U(A)$ is said to be an \emph{$\mathsf{A}$-arrow},
if there exists an arrow (necessarily unique) $a:A\to B$ such that
$U(a)=g$. An \emph{$U$-structured source} $S$ is a pair $(X,(f_{i},A_{i})_{i\in I})$
consisting of an object $X$ of $\mathsf{X}$ and a family of morphisms
$f_{i}:X\to U(A_{i})$ in $\mathsf{X}$, indexed by some set $I$.
Throughout this section, we write $U$-structured sources simply as
$S:=\{f_{i}:X\to U(A_{i})\}_{i\in I}$. Moreover, if $U$ is the identity
functor on $\mathsf{A}$, then $U$-structured sources are simply
referred to as \emph{sources}. A source $T=\{a_{i}:A\to A_{i}\}_{i\in I}$
is said to be a \emph{lifting of $S$} if $U(A)=X$ and $U(a_{i})=f_{i}$,
for every $i\in I$. Moreover, $T$ is said to be \emph{initial},
if for every arrow $g:U(B)\to U(A)$, it is the case that $g$ is
an $\mathsf{A}$-arrow provided that $f_{i}g:U(B)\to U(A_{i})$ is
an $\mathsf{A}$-arrow, for every $i\in I$. That is to say, if for
every $i\in I$, there exists a morphism $w_{i}:B\to A_{i}$ such
that $U(w_{i})=f_{i}g$, then there exists a morphism $h:B\to A$
such that $g=U(h)$. The converse is immediate. This situation is
depicted in the diagram below. 
\[
\begin{array}{ccc}
\xymatrix{B\ar@{-->}[rr]^-{\exists\;h}\ar@{-->}[dr]_-{\exists\;w_{i}} &  & A\ar[dl]^-{a_{i}}\\
 & A_{i}
}
 &  & \xymatrix{U(B)\ar[rr]^-{g=U(h)}\ar[dr]_-{U(w_{i})=f_{i}g} &  & U(A)\ar[dl]^-{f_{i}=U(a_{i})}\\
 & U(A_{i})
}
\end{array}
\]

Finally, a functor $U:\mathsf{A}\to\mathsf{X}$ is said to be \emph{topological}
if every $U$-structured source has a unique initial lifting. Therefore,
a concrete category $(\mathsf{A},U)$ over $\mathsf{X}$ is topological
if $U$ is topological. \\

Let $h:\langle\mathbf{A},\mathrm{C}\rangle\to\langle\mathbf{B},\mathrm{D}\rangle$
be a sequent homomorphism. It is evident that the assignments $U(\langle\mathbf{A},\mathrm{C}\rangle)=\mathbf{A}$,
$U(h)=h$ determine a faithful functor $U:\mathsf{SAB}\to\mathsf{BA}$.
So, the pair $(\mathsf{SAB},U)$ is a concrete category over $\mathsf{BA}$.

\begin{lemma}\label{existence of liftings} Let $\mathbf{A}$ be
a Boolean algebra and let $S=\{f_{i}:\mathbf{A}\to U(\langle\mathbf{B}_{i},\mathrm{C}_i\rangle)\}_{i\in I}$
be a $U$-structured source. Then $S$ has a lifting. \end{lemma} 
\begin{proof}
Set 
\[
\mathcal{F}_{S}=\{f^{-1}_{i}[F_{i}]\colon F_{i}\in X(\mathbf{B}_{i}),\;i\in I\}\cup\mathrm{Fi}(\mathbf{A}).
\]
For $Z\subseteq A$, define 
\[
\mathrm{D}_{S}(Z)=\bigcap\{H\in\mathcal{F}_{S}\colon Z\subseteq H\}
\]
and let 
\[
\mathrm{C}_{S}(Z)=\bigcup\{\mathrm{D}_{S}(W)\colon W\subseteq_{f}Z\}.
\]
The operator $\mathrm{D}_{S}$ is the closure operator generated by
the family $\mathcal{F}_{S}$, and $\mathrm{C}_{S}$ is its finitarization.
A standard argument shows that $\mathrm{C}_{S}$ is again a closure
operator: extensivity and monotonicity are immediate, and if $x\in\mathrm{C}_{S}(\mathrm{C}_{S}(Z))$,
only finitely many members of $\mathrm{C}_{S}(Z)$ are involved; collecting
finite witnesses for these members and using the idempotence of $\mathrm{D}_{S}$
gives $x\in\mathrm{C}_{S}(Z)$. Thus $\mathrm{C}_{S}$ is finitary
by construction.

We first show that, for every $i\in I$ and $Z\subseteq A$, 
\[
f_{i}[\mathrm{D}_{S}(Z)]\subseteq\mathrm{C}_i(f_{i}[Z]).
\]
Suppose that $x\in\mathrm{D}_{S}(Z)$ but $f_{i}(x)\notin\mathrm{C}_i(f_{i}[Z])$.
By Proposition \ref{clans for a basis}, there is $P\in X(\mathbf{B}_{i})$
such that $f_{i}[Z]\subseteq P$ and $f_{i}(x)\notin P$. Hence $Z\subseteq f^{-1}_{i}[P]\in\mathcal{F}_{S}$
while $x\notin f^{-1}_{i}[P]$, contradicting $x\in\mathrm{D}_{S}(Z)$.
Consequently, 
\[
f_{i}[\mathrm{C}_{S}(Z)]\subseteq\mathrm{C}_i(f_{i}[Z]),
\]
because every element of $\mathrm{C}_{S}(Z)$ already belongs to $\mathrm{D}_{S}(W)$
for some finite $W\subseteq Z$.

We verify the closure conditions. Since $\{1\}\in\mathrm{Fi}(\mathbf{A})$
and by \textbf{One} every member of $\mathcal{F}_{S}$ contains $1$, we have $\mathrm{C}_{S}(\emptyset)=\mathrm{D}_{S}(\emptyset)=\{1\}$,
so $\mathbf{One}$ holds. If $W\subseteq_{f}Z$, then $\mathrm{Fg}^{\mathbf{A}}(W)\in\mathcal{F}_{S}$, $\mathrm{D}_{S}(\mathrm{Fg}^{\mathbf{A}}(W))=\mathrm{Fg}^{\mathbf{A}}(W)$ and $W\subseteq \mathrm{Fg}^{\mathbf{A}}(W)$, whence $\mathrm{D}_{S}(W)\subseteq\mathrm{Fg}^{\mathbf{A}}(W)$.
Therefore 
\[
\mathrm{C}_{S}(Z)\subseteq\bigcup_{W\subseteq_{f}Z}\mathrm{Fg}^{\mathbf{A}}(W)\subseteq \bigvee_{W\subseteq_{f}Z}\mathrm{Fg}^{\mathbf{A}}(W)=\mathrm{Fg}^{\mathbf{A}}(Z),
\]
so $\mathbf{Fi}$ holds. Every member of $\mathcal{F}_{S}$ is an
upset; hence each $\mathrm{D}_{S}(W)$, and therefore $\mathrm{C}_{S}(Z)$,
is an upset. Thus $\mathbf{Up}$ holds.

It remains to verify $\mathbf{Jo}$. First note that $\mathrm{D}_{S}$
satisfies it. Suppose $x\in\mathrm{D}_{S}(Z,a)\cap\mathrm{D}_{S}(Z,b)$,
and let $H\in\mathcal{F}_{S}$ contain $Z\cup\{a\vee b\}$. If $H$
is a proper Boolean filter and $x\notin H$, an ultrafilter extending
$H$ and omitting $x$ would contain $a$ or $b$, contradicting the
assumed membership of $x$ in $\mathrm{D}_{S}(Z,a)\cap\mathrm{D}_{S}(Z,b)$.
(For $H=A$ there is nothing to prove.) If $H=f^{-1}_{i}[F_{i}]$,
then $f_{i}(a)\vee f_{i}(b)=f_{i}(a\vee b)\in F_{i}$; since the clan
$F_{i}$ is join-prime, either $a\in H$ or $b\in H$, and again $x\in H$.
Hence $x\in\mathrm{D}_{S}(Z,a\vee b)$.

Now let $x\in\mathrm{C}_{S}(Z,a)\cap\mathrm{C}_{S}(Z,b)$. There are
finite $W_{1},W_{2}\subseteq Z$ such that $x\in\mathrm{D}_{S}(W_{1},a)$
and $x\in\mathrm{D}_{S}(W_{2},b)$. With $W=W_{1}\cup W_{2}$, monotonicity
and the preceding paragraph yield 
\[
x\in\mathrm{D}_{S}(W,a)\cap\mathrm{D}_{S}(W,b)\subseteq\mathrm{D}_{S}(W,a\vee b)\subseteq\mathrm{C}_{S}(Z,a\vee b).
\]
Thus $\langle\mathbf{A},\mathrm{C}_{S}\rangle$ is a sequent algebra
and the maps $f_{i}$ are sequent homomorphisms. Hence $S$ has a
lifting. 
\end{proof}

\begin{lemma}\label{lifting is initial} Let $\mathbf{A}$ be a Boolean
algebra and let $S=\{f_{i}:\mathbf{A}\to U(\langle\mathbf{B}_{i},\mathrm{C}_i\rangle)\}_{i\in I}$
be a $U$-structured source. If $T_{S}$ denotes the lifting associated
with $S$ obtained in Lemma \ref{existence of liftings}, then $T_{S}$
is initial and unique. \end{lemma} 
\begin{proof}
Let $\langle\mathbf{E},\mathrm{D}\rangle$ be a sequent
algebra and let $g:\mathbf{E}\to\mathbf{A}$ be a Boolean homomorphism 
such that every composite $f_{i}g$ is a sequent homomorphism. We
prove that $g$ is a sequent homomorphism. Let $W\subseteq E$ and
$x\in\mathrm{D}(W)$. By finitarity there is a finite
$W_{0}\subseteq W$ such that $x\in\mathrm{D}(W_{0})$.
It suffices to show that $g(x)\in\mathrm{D}_{S}(g[W_{0}])$.

Let $H\in\mathcal{F}_{S}$ and suppose $g[W_{0}]\subseteq H$. If
$H\in\mathrm{Fi}(\mathbf{A})$, then $\mathbf{Fi}$ gives $x\in\mathrm{Fg}^{\mathbf{E}}(W_{0})$. Since $g$ is a Boolean homomorphism, $g(x)\in\mathrm{Fg}^{\mathbf{A}}(g[W_{0}])\subseteq H$.
If instead $H=f^{-1}_{i}[F_{i}]$ for a clan $F_{i}\in X(\mathbf{B}_{i})$,
then $(f_{i}g)[W_{0}]\subseteq F_{i}$. Since $f_{i}g$ is a sequent
homomorphism, 
\[
(f_{i}g)(x)\in\mathrm{C}_{\mathbf{B}_{i}}((f_{i}g)[W_{0}])\subseteq F_{i},
\]
so again $g(x)\in H$. Therefore $g(x)$ belongs to every member of
$\mathcal{F}_{S}$ that contains $g[W_{0}]$, and hence 
\[
g(x)\in\mathrm{D}_{S}(g[W_{0}])\subseteq\mathrm{C}_{S}(g[W]).
\]
Thus $g$ is a sequent homomorphism, so the lifting is initial.

For uniqueness, if $\mathrm{C}$ gives another initial lifting on
the same Boolean algebra $\mathbf{A}$, initiality applied to the
identity map in both directions shows that $\mathrm{C}(Z)\subseteq\mathrm{C}_{S}(Z)$
and $\mathrm{C}_{S}(Z)\subseteq\mathrm{C}(Z)$ for every $Z\subseteq A$.
Hence $\mathrm{C}=\mathrm{C}_{S}$. 
\end{proof}

Hence, as a consequence of Lemmas \ref{existence of liftings} and
\ref{lifting is initial} we conclude:

\begin{theorem}\label{SeqSH is topological over Boole} $\mathsf{SAB}$
is topological over $\mathsf{BA}$. \end{theorem}

\begin{corollary} The functor $U:\mathsf{SAB}\to\mathsf{BA}$ has
both a left adjoint ($L$) and a right adjoint ($R$)\textcolor{red}{,}
which are full embeddings. Furthermore, $UL=Id_{\mathsf{BA}}=UR$.
\end{corollary} 
\begin{proof}
From Theorem \ref{SeqSH is topological over Boole} and Proposition
21.12 of \cite{AHS2006}, the result holds. 
\end{proof}

Let $(\mathsf{A},U)$ be concrete over $\mathsf{X}$ and let $Y$
be an object of $\mathsf{X}$. We recall that the \emph{fibre} of
$Y$ is the preordered class consisting of all objects $B$ of $\mathsf{A}$
such that $U(B)=Y$, ordered by: 
\[
B\leq C\quad\text{if and only if}\quad\mathrm{id}_{Y}:U(B)\to U(C)\text{ is an }\mathsf{A}\text{-morphism}.
\]
We say that $(\mathsf{A},U)$ is \emph{fibre-complete} if each of
its fibres is a complete lattice. Furthermore, if $(\mathsf{A},U)$
is topological over $\mathsf{X}$, then one of the immediate consequences
is that $(\mathsf{A},U)$ is fibre-complete (Proposition 21.11 of
\cite{AHS2006}). In our setting, the fibre over a Boolean algebra
can be identified with the preordered set of all the finitary closure
operators satisfying \textbf{Fi}, \textbf{One}, \textbf{Up}, and \textbf{Jo}
that can be defined on it, preordered by pointwise inclusion. Therefore,
by Theorem \ref{SeqSH is topological over Boole} we obtain:

\begin{proposition}\label{fibre-complete} The poset of all finitary
closure operators on a Boolean algebra $\mathbf{A}$ satisfying $\mathbf{Fi}$,
$\mathbf{One}$, $\mathbf{Up}$, and $\mathbf{Jo}$ is a complete
lattice under pointwise inclusion. Its greatest element is $\mathrm{Fg}^{\mathbf{A}}$,
and its least element is the operator 
\[
\mathrm{C}_{\min}(Z)=\begin{cases}
\{1\}, & Z=\emptyset,\\
\uparrow Z, & Z\neq\emptyset.
\end{cases}
\]
\end{proposition} 
\begin{proof}
Let $\{\mathrm{C}_{i}\}_{i\in I}$ be a family of such operators and
define first 
\[
\mathrm{D}(Z)=\bigcap_{i\in I}\mathrm{C}_{i}(Z),\qquad\mathrm{C}_{\wedge}(Z)=\bigcup_{W\subseteq_{f}Z}\mathrm{D}(W).
\]
The operator $\mathrm{D}$ is a closure operator satisfying $\mathbf{Fi}$,
$\mathbf{One}$, $\mathbf{Up}$, and $\mathbf{Jo}$, although it need
not be finitary. Its finitarization $\mathrm{C}_{\wedge}$ is a finitary
closure operator and retains all four properties; the verification
of $\mathbf{Jo}$ uses the union of the two finite witnesses. Moreover,
$\mathrm{C}_{\wedge}\leq\mathrm{C}_{i}$ for every $i$, and every
finitary closure operator below all the $\mathrm{C}_{i}$ is below
$\mathrm{C}_{\wedge}$. Thus arbitrary meets exist. Since $\mathrm{Fg}^{\mathbf{A}}$
is a greatest element, the poset is a complete lattice.

The maximality of $\mathrm{Fg}^{\mathbf{A}}$ follows immediately
from $\mathbf{Fi}$. For the least element, $\mathrm{C}_{\min}$ is
readily checked to be a finitary closure operator satisfying all four
conditions. If $Z\neq\emptyset$ and $y\in\uparrow Z$, choose $z\in Z$
with $z\leq y$. Extensivity gives $z\in\mathrm{C}(Z)$, and $\mathbf{Up}$
yields $y\in\mathrm{C}(Z)$. If $Z=\emptyset$, $\mathbf{One}$ gives
$\mathrm{C}_{\min}(\emptyset)=\{1\}=\mathrm{C}(\emptyset)$. Hence
$\mathrm{C}_{\min}\leq\mathrm{C}$ for every sequent closure operator
$\mathrm{C}$. 
\end{proof}

\begin{remark} We stress that Proposition \ref{fibre-complete} allows
us to identify the functors $L$ and $R$ explicitly. Let $g:\mathbf{A}\to\mathbf{B}$
be a morphism of Boolean algebras. The left adjoint of $U$ (usually
called the \emph{discrete functor}) is defined by $L(\mathbf{A})=\langle\mathbf{A},\mathrm{C}_{\min}\rangle$
and $L(g)=g$. Further, the right adjoint of $U$ (usually called
the \emph{indiscrete functor}) is defined by $R(\mathbf{A})=\langle\mathbf{A},\mathrm{Fg}^{\mathbf{A}}\rangle$
and $R(g)=g$.

\end{remark}

The significance of Theorem \ref{SeqSH is topological over Boole}
lies in the following fact: if $(\mathsf{A},U)$ is topological over
$\mathsf{X}$, then $\mathsf{A}$ inherits many categorical properties
of $\mathsf{X}$. Thus\textcolor{red}{,} since $\mathsf{BA}$ is complete
and cocomplete, by Theorem 21.16 of \cite{AHS2006}, the following
hold.

\begin{corollary} The category $\mathsf{SAB}$ is complete and
cocomplete. \end{corollary}

Let $\mathsf{BSp}$ be the category of Boolean spaces and continuous
functions. As a consequence of Theorem VI.21.9 of \cite{AHS2006},
the well-known Stone duality for Boolean algebras and Theorem \ref{dually category},
we obtain:

\begin{corollary}\label{SeqSp is topological over BSp} $\mathsf{SeqSp}_{\mathrm{m}}$
is topological over $\mathsf{BSp}$. Moreover, it is complete and
cocomplete. \end{corollary}

Let $\mathsf{BASH}$ be the category whose objects are Boolean algebras and whose morphisms are the join-semihomomorphisms preserving $1$ between them. We conclude this section by noting that one may be tempted to replicate the process described above by considering the faithful functor $V\colon\mathsf{SA}\to\mathsf{BASH}$ determined by the assignments $V(\langle\mathbf{A},\mathrm{C}\rangle)=\mathbf{A}$ and $V(h)=h$, where $h\colon\langle\mathbf{A},\mathrm{C}\rangle\to\langle\mathbf{B},\mathrm{D}\rangle$ is a sequent semi-homomorphism, in order to obtain a result analogous to Theorem~\ref{SeqSH is topological over Boole}. In fact, it is not hard to see that the same arguments used in Lemma~\ref{existence of liftings} show that every $V$-structured source has a lifting. Nevertheless, we point out that this is as far as we can go in this direction, as the following result shows.

\begin{proposition}\label{prop: SA not topological over BASH}
$\mathsf{SA}$ not topological over $\mathsf{BASH}$.    
\end{proposition}
\begin{proof}
Let $\mathbf A=\{0,a,b,1\}$ be the four-element Boolean algebra and
$g\colon \mathbf A\to V\langle\mathbf A,\mathrm{Fg}^{\mathbf A}\rangle$ the map with
$g(0)=0$ and $g(x)=1$ for $x\neq 0$. Since $g(x\vee y)=0$ iff $x=y=0$,
$g$ preserves finite joins and $1$, so it is a $\mathsf{BASH}$-morphism.

Observe that 
\[
\mathcal{F}_{g}=\{g^{-1}[F]\colon F\in X(\mathbf{A})\}\cup\mathrm{Fi}(\mathbf{A})=\{A, \{a,b,1\},\{a,1\}, \{b,1\}, \{1\}\},
\]
and since $\mathbf{A}$ is finite, for every $Z\subseteq A$ we get 
\[
\mathrm{C}_{g}(Z)=\bigcap\{H\in\mathcal{F}_{g}\colon Z\subseteq H\}=
\begin{cases}
A, & \text{if } 0\in Z,\\[2pt]
\{a,b,1\}, & \text{if } 0\notin Z \text{ and } a,b\in Z,\\[2pt]
\{a,1\}, & \text{if } a\in Z \text{ and } b,0\notin Z,\\[2pt]
\{b,1\}, & \text{if } b\in Z \text{ and } a,0\notin Z,\\[2pt]
\{1\}, & \text{if } Z\subseteq\{1\}.
\end{cases}
\]
Replicating the proof of Lemma~\ref{existence of liftings} (which only uses that the maps $f_i$ preserve finite joins and $1$),
$\langle\mathbf A,\mathrm C_g\rangle$ is a sequent algebra and
$g\colon\langle\mathbf A,\mathrm C_g\rangle\to U\langle\mathbf A,\mathrm{Fg}^{\mathbf A}\rangle$ is a sequent semi-homomorphism.
We claim that $\mathrm C_g$ is the only possible initial lifting of $\{g\}$.
Let $\mathrm C$ be any closure operator on $A$ such that $\langle\mathbf A,\mathrm C\rangle$ is a sequent algebra and
$g\colon\langle\mathbf A,\mathrm C\rangle\to U\langle\mathbf A,\mathrm{Fg}^{\mathbf A}\rangle$ is a sequent semi-homomorphism.
If $0\notin Z$, then $g[Z]\subseteq\{1\}$, so $g[\mathrm C(Z)]\subseteq\mathrm{Fg}^{\mathbf A}(g[Z])=\{1\}$, whence $0\notin\mathrm C(Z)$.
Together with $\mathbf{Fi}$ this gives $\mathrm C(Z)\subseteq\mathrm{Fg}^{\mathbf A}(Z)\setminus\{0\}\subseteq\mathrm C_g(Z)$,
and if $0\in Z$ then $\mathrm C_g(Z)=A\supseteq\mathrm C(Z)$. Hence $\mathrm C\subseteq\mathrm C_g$.

Now suppose $\mathrm C$ is an initial lifting of $\{g\}$. Since $\mathrm{id}_A$ is a $\mathsf{BASH}$-morphism and
$g\circ\mathrm{id}_A=g$ is a sequent semi-homomorphism out of $\langle\mathbf A,\mathrm C_g\rangle$,
initiality gives $\mathrm C_g\subseteq\mathrm C$. Thus $\mathrm C=\mathrm C_g$.
It remains to show that $\mathrm C_g$ is not initial. To do so, let $P=\{p,q,r\}$ and $\mathbf E$ be the Boolean algebra whose universe is $\mathcal P(P)$, with $0=\varnothing$ and $1=P$. Now we define for every $W\subseteq E$:
\[
\mathrm{D}(W)=
\begin{cases}
E, & \text{if } \varnothing\in W,\\[2pt]
\uparrow (\bigcap W)\setminus\{\varnothing\}, & \text{if } \varnothing\notin W.
\end{cases}
\]
\emph{Claim 1.} $\langle\mathbf{E},\mathrm{D}\rangle$ is a sequent algebra. 
\medskip

\noindent
By Proposition~\ref{prop: charact closure} we check that $\mathrm D$ is a finitary closure operator
satisfying $\mathbf{Fi}$, $\mathbf{One}$, $\mathbf{Up}$, $\mathbf{Jo}$.
Note that $\mathrm{Fg}^{\mathbf E}(W)=\uparrow (\bigcap W)$, so
$\mathrm D(W)=\mathrm{Fg}^{\mathbf E}(W)$ if $\varnothing\in W$ and
$\mathrm D(W)=\mathrm{Fg}^{\mathbf E}(W)\setminus\{\varnothing\}$ otherwise. To prove that $\mathrm{D}$ is a closure operator, we only show idempotency since the remainig conditions are easy. Let $K=\mathrm D(W)$. If $\varnothing\in W$, then $K=E$ and $\mathrm D(K)=E=K$.
Suppose $\varnothing\notin W$. Since $W\subseteq K$, we have $\bigcap K\subseteq\bigcap W$.
Conversely, every element of $K$ contains $\bigcap W$ and $K\neq\varnothing$ (it contains $P$),
so $\bigcap W\subseteq\bigcap K$. Thus $\bigcap K=\bigcap W$, and since $\varnothing\notin K$,
\[
\mathrm D(K)=\uparrow\!\bigl(\textstyle\bigcap K\bigr)\setminus\{\varnothing\}
=\uparrow\!\bigl(\textstyle\bigcap W\bigr)\setminus\{\varnothing\}=\mathrm D(W)=K.
\]

Observe that $\mathbf{Fi}$ holds by the remark above;
$\mathrm D(\varnothing)=\{P\}$ gives $\mathbf{One}$; and $\mathbf{Up}$ holds because
$\varnothing$ is the least element, so removing it from an up-set leaves an up-set. Finally, we prove \emph{$\mathbf{Jo}$.} Let $y\in\mathrm D(X,c)\cap\mathrm D(X,d)$ and $m=\bigcap X$.
If $\varnothing\in X$, then $\mathrm D(X,c\cup d)=E$. If $c=\varnothing$ (resp.\ $d=\varnothing$),
then $c\cup d=d$ (resp.\ $c$) and there is nothing to prove.
Otherwise $\varnothing\notin X\cup\{c\cup d\}$, $y\neq\varnothing$, and $m\cap c\subseteq y$, $m\cap d\subseteq y$,
so $m\cap(c\cup d)=\bigcap(X\cup\{c\cup d\})\subseteq y$. Hence $y\in\mathrm D(X,c\cup d)$.
\medskip

Define $h\colon\mathbf E\to\mathbf A$ by
\[
h(X)=\bigvee_{t\in X}h(\{t\}),\qquad h(\{p\})=h(\{q\})=a,\quad h(\{r\})=b,\quad h(\varnothing)=0 .
\]  
Then $h$ preserves finite joins and $h(P)=a\vee b=1$, so $h$ is a $\mathsf{BASH}$-morphism, and
$h(X)=0$ iff $X=\varnothing$. Note that $h$ does not preserve meets:
$h(\{p\})\wedge h(\{q\})=a$ while $h(\{p\}\cap\{q\})=0$.
\medskip

\noindent
\emph{Claim 2.} $g h$ is an sequent semi-homomorphism $\langle\mathbf E,\mathrm{D}\rangle\to\langle\mathbf A,\mathrm{Fg}^{\mathbf A}\rangle$.
\medskip

\noindent
It is a $\mathsf{BASH}$-morphism as a composite of such, and $g h(X)=0$ iff $X=\varnothing$.
If $\varnothing\in W$, then $0\in gh[W]$, so $\mathrm{Fg}^{\mathbf A}(gh[W])=A$ and there is nothing to prove.
If $\varnothing\notin W$, then $\varnothing\notin\mathrm{D}(W)$, so $gh[\mathrm{D}(W)]\subseteq\{1\}\subseteq \mathrm{Fg}^{\mathbf A}(gh[W])$.
\medskip

\noindent
\emph{Claim 3}: $h\colon\langle\mathbf E,\mathrm{D}\rangle\to\langle\mathbf A,\mathrm{C}_g\rangle$ is not an sequent semi-homomorphism.
\medskip

\noindent
Let $W=\{\{p\},\{q\}\}$. Then $\varnothing\notin W$ and $\{p\}\cap\{q\}=\varnothing$, so
$\mathrm{Fg}^{\mathbf E}(W)=E$ and $\{r\}\in E\setminus\{\varnothing\}=\mathrm{D}(W)$.
On the other hand $h[W]=\{a\}$, so $\mathrm{C}_g(h[W])=\mathrm{Fg}^{\mathbf A}(\{a\})=\{a,1\}$,
whereas $h(\{r\})=b\notin\{a,1\}$. Hence $h[\mathrm{D}(W)]\not\subseteq \mathrm{C}_g(h[W])$.

\medskip
\noindent By Claims 2 and 3, $h$ is a $\mathsf{BASH}$-morphism such that $gh$ is a sequent semi-homomorphism,
but $h\colon\langle\mathbf E,\mathrm D\rangle\to\langle\mathbf A,\mathrm C_g\rangle$ is not.
So $\mathrm C_g$ is not initial, and since it is the only candidate, $\{g\}$ has no initial lifting.
Hence $V$ is not topological.
\end{proof}

\section{Acknowledgements} 
This research has been funded by the CONICET [PIP 11220250100318CO] and the European Research
Council (ERC) under the European Union’s Horizon 2025 research and innovation program (grant agreement No.101299559).

\section{Declaration of generative AI and AI-assisted technologies in the manuscript preparation process}
During the preparation of this work, the authors used ChatGPT (OpenAI) and Claude (Anthropic) to improve the English grammar and style, refine the mathematical exposition, search for counterexamples, and assist with \LaTeX{} formatting. The authors reviewed and edited the content as needed and take full responsibility for the published article.

\end{document}